\documentclass[11pt]{article}
\usepackage[utf8]{inputenc} 
\usepackage[T1]{fontenc} 
\usepackage{amsmath}
\usepackage{graphicx}
\usepackage{amsfonts}
\usepackage{amssymb}
\usepackage{xcolor}
\usepackage{hyperref}
\usepackage{amsthm}
\usepackage{amscd}
\usepackage{mathrsfs}
\usepackage{float}
\usepackage{appendix}
\usepackage{subfigure}
\usepackage{glossaries}
\usepackage{stmaryrd}
\usepackage{makeidx}
\usepackage{geometry}
\usepackage{enumerate}
\usepackage{authblk}

\newtheorem{thm}{Theorem}[section]

\newtheorem{corollary}[thm]{Corollary}
\newtheorem{proposition}[thm]{Proposition}
\newtheorem{prop}[thm]{Proposition}
\newtheorem{lemma}[thm]{Lemma}

\newtheorem{remark}[thm]{Remark}
\newtheorem{example}[thm]{Example}
\newtheorem{definition}[thm]{Definition}
\newtheorem{assumption}{Assumption}

\newcommand{\norm}[1]{\|{#1}\|}

\renewcommand{\Phi}{\varPhi}
\newcommand{\R}{{\mathbb{R}}}

\newcommand{\C}{{\mathbb{C}}}

\newcommand{\N}{{\mathbb{N}}}
\DeclareMathOperator{\GL}{GL}

\definecolor{dkgreen}{rgb}{0,0.4,0}
\newcommand{\MK}[1]{{\color{dkgreen}{#1}}}
\definecolor{dkgreen}{rgb}{0,0.4,0}

\usepackage[normalem]{ulem}

\makeindex

\begin{document}
\title{Symmetry reduction and recovery of trajectories of optimal control problems via measure relaxations} 
 
\author[1]{Nicolas Augier}
\author[2]{Didier Henrion}
\author[2]{Milan Korda}
\author[2]{Victor Magron}
\affil[1]{\footnotesize 
Mac Team, LAAS-CNRS, Toulouse, France
}
\affil[2]{\footnotesize 
Pop Team, LAAS-CNRS, Toulouse, France
}

    \maketitle
    \begin{abstract}
We address the problem of symmetry reduction of optimal control problems under the action of a finite group from a  measure relaxation viewpoint. We propose a method based on the moment-SOS aka Lasserre hierarchy which allows one to significantly reduce the computation time and memory requirements compared to the case without symmetry reduction. We show that the recovery of optimal trajectories boils down to solving a symmetric parametric polynomial system. 
Then we illustrate our method on the symmetric integrator and the time-optimal inversion of qubits.
    \end{abstract}

\section{Introduction}

While symmetries and associated computer algebra methods have been extensively studied in the framework of dynamical systems~\cite{CL00,G00,HUBERT,FANTUZZI20}, few results have been obtained in this direction for control systems.
The need for efficient algorithms providing global solutions to optimal control is a key challenge. A significant effort has been put into exploiting symmetries in extremal trajectories with the Pontryagin Maximum Principle (see~\cite{OHSAWA_SYMMETRIES,SUSSMANN_SYMMETRIES}). 
In particular, in~\cite{OHSAWA_SYMMETRIES} the author defines a reduced Hamiltonian system via Poisson reduction for Lie group symmetries. However, to the best of our knowledge, no efficient numerical method has been implemented in this setting.
On the other hand, symmetry reduction has been successfully applied to problems of static constrained polynomial optimization~\cite{SYMSDP}, especially when the optimization problem is invariant under the action of a finite group. 
It turned out that a symmetry adapted numerical solution allows one to significantly reduce the computation effort without compromising accuracy.
The literature on exploiting symmetries for globally solving optimal control problems is much sparser. In financial mathematics, symmetries of the Hamilton-Jacobi-Bellman (HJB) equation attached to specific stochastic processes were exploited in \cite{nal05,los07}. In \cite{rhc16} the authors are exploiting symmetries of a particular class of optimal control problems with the objective of generating benchmark HJB equations with explicit analytic solutions. Symmetries in an infinite-dimensional linear programming relaxation of the problem of extreme value computation were exploited in~\cite{FANTUZZI20}.

This paper complements recent efforts to improve the scalability of the moment-SOS hierarchy to perform more efficient analyses of dynamical systems with sparse input data~\cite{schlosser2020sparse,wang2021sparsejsr,wang2021exploiting}; see also the recent survey\MK{s} \MK{\cite{sparsebook, zheng2021chordal}} describing several sparsity exploiting techniques for polynomial optimization problems.
Here, we develop a systematic method of symmetry reduction for optimal control problems, when the symmetry is induced by a finite group (infinite symmetry groups are beyond the scope of the present paper), using the measure relaxation formulation introduced in~\cite{vinter1993} and its numerical solution via
the moment-SOS (sum of squares) aka Lasserre hierarchy, as originally proposed in~\cite{OCP08}.
In particular, we show that the solution of the optimal control problem can be reduced to solving semi-definite programming (SDP) problems possessing symmetry invariance properties.
To the best of the authors' knowledge, the problem of symmetry reduction of optimal control problems has not been studied yet with measure relaxations.
The strength of our approach is that it provides a globally converging method for nonlinear optimal control via the solution of convex SDP problems. The symmetry reduction allows one to compute lower bounds on the optimal cost in a more efficient way, which is a crucial task for  applications, including those from quantum control~\cite{gatto}: among them one can cite the computation of the so-called \emph{quantum speed limit}~\cite{QSL}, or biological systems, which are often composed of large-scale networks of structurally similar dynamical systems having symmetries \cite{GOLUB,GOLUB2}.

Another contribution of this paper is to propose a method allowing one to recover the moments of optimal trajectories.
In the polynomial optimization case without symmetry reduction,  specific tools for extraction of minimizers have been developed in~\cite{Henrion2005} based on flat extensions~\cite{curto2000truncated}. The distinguishing feature of this method is the possibility to extract {all} global minimizers provided there are only finitely many of them and certain genericity conditions are satisfied. In the optimal control case, the question of recovery of optimal trajectories is much more intricate, mainly due to the fact that the dynamical setting provides an infinite number of moment conditions for which the flat extension condition  cannot be satisfied.
The non-uniqueness of optimal trajectories for an optimal control problem is a classical feature (see for instance \cite{SUBRIEM} for an overview in the sub-Riemannian setting), and is related to the notion of the so-called \emph{cut locus}.
We will see along the article that this phenomenon appears naturally for optimal control problems admitting symmetries and is an obstacle to the recovery of optimal trajectories when using the moment-SOS hierarchy. This difficulty has already been underlined in the context of Generalized Moment Problem in~\cite[Remark 13]{TACCHI22}.
We propose here a practical reconstruction method of the state and control trajectories in the case of systems having symmetries, both using invariant polynomials and the characterization of occupation measures as extreme points of the set of solution of the relaxed problem given in~\cite{vinter1993}. 

\subsection{Contributions}

The main results of the paper can be summarized as follows:
\begin{enumerate}
\item We adapt the symmetry-adapted SDP scheme developed in~\cite{SYMSDP} to the optimal control setting, with convergence guarantees, allowing one to provide more precise lower bounds on the cost than  without exploiting symmetry and at a lower computational cost.
\item We propose a method of approximation of optimal trajectories by using invariant polynomials and a selection algorithm of occupation measures.  Our algorithm can be decomposed into the resolution of two successive symmetry-reduced moment-SOS hierarchies for which the uniqueness of the solution is guaranteed.
\end{enumerate}
We illustrate theses results by an efficient recovery of time-optimal trajectories for two-level quantum systems, using  a symmetry-adapted moment-SOS hierarchy.

\section{Notations}
\begin{itemize}
\item For $n\geq 1$, denote the group of invertible $n\times n$ matrices by $GL_n(\R)$.
Denote the direct sum of two subspaces $E_1$ and $E_2$ of a vector space $V$ by $E_1\bigoplus E_2$, and for two matrices $(M_1,M_2)\in \GL_n(\R)\times \GL_m(\R)$ with $n,m\geq 1$, denote the $(n+m) \times (n+m)$ block matrix composed of the two blocks $M_1$ and $M_2$ by $M_1\bigoplus M_2$.
\item We let $L^\infty(X,S)$ denote the Banach space of essentially bounded measurable functions on a set $X$ with values in a set $S$ with the essential supremum.
For a $C^1$ function $\varphi$ on a product set $X\times U$, let $\nabla_x \varphi$ be gradient of $\varphi$ w.r.t. the variable $x\in X$.
\item A group $G$ with identity element $e$ is said to act on a set $X$ when there is a map $\psi:G\times X \to X$ such that the following holds for every $x\in X$:
\begin{enumerate}
    \item $\psi(e,x)=x$
    \item $\psi(g,\psi(h,x))=\psi(gh,x)$, for every $g,h\in G$.
\end{enumerate}
By a slight abuse of notations, we will write in this article $\psi(g,x)=g(x)$, for every $(g,x)\in G\times X$.
A group homomorphism between two groups $G$ and $H$ is a mapping $\tau :G\to H$ such that $\tau (g_1 g_2)=\tau (g_1)\tau (g_2)$, for every $g_1,g_2\in G$.
\item Denote the set of  non-negative Borel measures on a set $S$ by $\mathcal{M}_+(S)$.
We say that a sequence $(\mu_k)_{k\in \N}$ of measures in $\mathcal{M}_+(S)$  weak-$\star$ converges to $\mu\in \mathcal{M}_+(S)$, if for every continuous function $\varphi$ on $S$ with compact support, we have $\int_S \varphi(x) d\mu_k(x) \to \int_S \varphi(x) d\mu(x)$, when $k\to \infty$.

\item For $n\geq 1$, denote the set of real polynomials with indeterminates $x_1,\dots,x_n$ by $\R[x_1,\dots,x_n]$.
For $\alpha=(\alpha_1,\dots,\alpha_n)\in \N^n$, define $x^\alpha=x_1^{\alpha_1}\dots x_n^{\alpha_n}$.
\item For a subset $S$ of a vector space $V$, denote the convex hull of $S$ by $\text{Conv}(S)$, i.e. the smallest convex set included in $V$ containing $S$.
\end{itemize}

\section{Definitions and basic facts}\label{FRAME}
 Consider a polynomial map $f$ in the variables $(x,u)\in X\times U$ with values in $\R^n$, where $X$ is a compact semi-algebraic subset of $\R^n$ and $U$ is a compact semi-algebraic subset of $\R^m$ with $n,m\geq 1$, and consider the controlled ordinary differential equation
\begin{equation}\label{general}
\dot{x}(t)=f(x(t),u(t)),
\end{equation}
where $t\mapsto x(t)$ takes values in $X$, and  the control $t\mapsto u(t)$ takes values in $U$.

Let \[X=\{x\in \R^n \mid \forall j\in \{1,\dots,p\}, \  v_j(x)\geq 0\},\] \[K=\{x\in \R^n \mid \forall j\in \{1,\dots,q\}, \  \theta_j(x)\geq 0\},\]  where $(v_j)_{j\in \{1,\dots,p\}}$, $(\theta_j)_{j\in \{1,\dots,q\}}$ with $p,q\geq 1$ are families of polynomials w.r.t. the variable $x$, such that $K\subseteq X$, and 
\[U=\{u\in \R^m \mid \forall j\in \{1,\dots,l\}, \  w_j(u)\geq 0\},\] where $(w_j)_{j\in \{1,\dots,l\}}$ with $l\geq 1$ is a family of polynomials w.r.t. the variable $u$.

Following the definition given in~\cite{OHSAWA_SYMMETRIES}, we define a symmetry for control system~\eqref{general} as follows.
\begin{definition}\label{DEF_SYM}
For a subgroup $G$ of $\GL_n(\R)$ acting on $X$, and a subgroup $H$ of $\GL_m(\R)$ acting on $U$, we say that Equation~\eqref{general} is $G$-\emph{invariant} if there exists a group homomorphism $\tau:G\to H$
such that for every $g\in G$, $f(g(x),\tau(g)(u))=g(f(x,u))$ for every $(x,u)\in X\times U$. 
\end{definition}
A simple example 
of such invariance is the case of a controlled harmonic oscillator corresponding to $f(x_1,x_2,u)=\begin{pmatrix}
x_2\\
-x_1+u\\
\end{pmatrix}$. We have that Equation~\eqref{general} is $G$-invariant with $G=\{\text{Id}_{\R^2},-\text{Id}_{\R^2}\}$, $H = \{1,-1\}$ and $\tau(\pm \text{Id}_{\R^2})=\pm 1$.
\begin{definition}\label{group_inv_def}
A set $S$ is $G$-invariant if $g(S) \subset S$ for all $g \in G$.
\end{definition}


Assume that the sets $X,K,U,\{x_0\}$ and Equation~\ref{general} are $G$-invariant~\footnote{The requirement of $x_0$ being $G$-invariant can be relaxed to $x_0 \in X_0$ with $X_0$ being $G$-invariant.}, and consider the optimal control problem:

\begin{equation}
\tag{OCP}
\label{optimalcontrol}
\begin{aligned}
\rho:=&\text{inf}_{u\in L^{\infty}([0,T],U)}  \quad  J(x,u)=\int_0^T h(t,x(t),u(t)) dt + H(x(T)) \\
 &\textrm{s.t.} \   (x(t),u(t))  \ \textrm{satisfies~equation~\eqref{general}}, \
x(0)=x_0, \ x(T)\in K, 
\end{aligned}
\end{equation}
where $h : [0,T]\times X \times U \to \R$ is a polynomial function  such that $h(t,g(x),\tau(g)(u))=h(t,x,u)$, for every $(t,x,u)\in [0,T]\times X \times U$, and $H : X \to \R$ is a polynomial function such that $H(g(x))=H(x)$, for every $g\in G$, $x\in X$.


\subsection{Embedding of the optimal control problem into a linear program on measures}

We relax Problem~\eqref{optimalcontrol} into a linear program on measures, following the approach pioneered in~\cite{lewis1980relaxation,rubio1975generalized,rubio1976extremal,vinter1978equivalence} and used computationally in conjunction with semi-definite programming (SDP) hierarchies in~\cite{OCP08}.

\subsubsection{Liouville Equation}

Given a test function $\phi \in C^1([0,T]\times \mathbb{R}^n)$ and a trajectory $x(\cdot)$ of (\ref{general}) generated by a control input $u(\cdot)$ starting from the initial condition $x_0$, we have
\begin{equation}\label{eq:tmpliouv}
\phi(T,x(T)) - \phi(0,x_0) = \int_0^T \dot{\phi}(t,x(t)) \,dt =  \int_0^T \frac{\partial \phi}{\partial t}(t,x(t)) + \left\langle \nabla_x\phi(t,x(t)),f(x(t),u(t))\right\rangle\,dt.
\end{equation}
To each trajectory-control pair, we can associate a measure $\mu$ associated to the trajectory-control pair $(x(t),u(t))$ is defined by
\[
\mu(A\times B \times C) = \int_0^T I_{A\times B\times C}(t,x(t),u(t))\,dt,
\]
for all Borel sets $A\subset [0,T]$, $B \subset \mathbb{R}^n$, $C\subset U$,
and a terminal measure defined by
\[
\mu_T = \delta_{x(T)}.
\]

\begin{definition}\label{OCCUP_MEAS}
We say that a pair of measures $(\nu,\nu_T)\in \mathscr{M}_+([0,T] \times X \times U)\times \mathscr{M}_+(K)$ is a \emph{pair of occupation measures} associated with a trajectory of Equation~\eqref{general} if $d\nu(t,x,u)=dt \delta_{x(t)}(dx)\delta_{u(t)}(du)$ and $d\nu_T(x)=\delta_{x(T)}(dx)$, for some pair $(x(t),u(t))$ satisfying Equation~\eqref{general}, where $x(\cdot):[0,T]\to X$ is a Lipschitz function and $u(\cdot)\in L^\infty([0,T],U)$.
\end{definition}

For a pair $(\mu,\mu_T)$ of occupation measures, the definition of $\mu$ implies that
\[
\int_0^T h(t,x(t),u(t))\,dt = \int_{[0,T]\times \mathbb{R}^n\times U} h(t,x,u)\,d\mu(t,x,u) ,
\]
for all $h \in L_1(\mu)$. The measure $\mu$ therefore uniquely encodes the trajectory-control pair $(x(t),u(t))$ and allows one to replace integration along the trajectory-control pair by spatial integration with respect to $\mu$. Using this and~(\ref{eq:tmpliouv}), we obtain the relation
\begin{equation}\label{Liouville}
\int \phi(T,x)\,d\mu_T(x) -\phi(0,x_0) = \int_{[0,T]\times \mathbb{R}^n\times U} \frac{\partial \phi}{\partial t} + \left\langle \nabla_x\phi(t,x), f(x,u)\right\rangle\,d\mu(t,x,u).
\end{equation}
This equation is referred to as the \emph{Liouville equation}. Crucially, the equation is affine in the measures $(\mu,\mu_T)$.

\subsubsection{Relaxed OCP}
For a set $S$ in a Euclidean space, denote the set of Borel measures defined on $S$ by $\mathscr{M}_+(S)$.
If the infimum of \eqref{optimalcontrol} is attained, denote the set of optimal pairs by $\mathcal{T}=\{(x(t),u(t))\in X\times U \ \text{optimal in} \ \eqref{optimalcontrol}\}$.
With these definitions, we are ready to write an infinite-dimensional linear programming (LP) relaxation of Problem~(\ref{optimalcontrol}). The relaxation reads
\begin{equation} \label{opt:relaxed}\tag{$\mathcal{O}$}
\begin{aligned}
\rho^\star = \inf\limits_{\mu,\mu_T} \quad & \int h\,d\mu+ \int H\,d\mu_T\\
\textrm{s.t.} \quad &   (\mu,\mu_T) \text{\;\,satisfy (\ref{Liouville})} \;\forall\,\phi\in C^1([0,T]\times X)\\
  &(\mu,\mu_T) \in \mathscr{M}_+([0,T]\times X\times U) \times \mathscr{M}_+(K) .   \\
\end{aligned}
\end{equation}

We observe that the objective function of this optimization problem is linear, the first constraint is affine and the last constraint is an inclusion into a convex cone. Problem~\eqref{opt:relaxed} is therefore an infinite-dimensional linear programming problem with the decision variables $(\mu,\mu_T)$. Note that the constraint sets on the trajectories and controls are imposed through the conic inclusion (the last constraint) of the optimization problem.

\begin{definition} Let us denote by $\mathcal{M}$ the convex set of optimal solutions $(\mu,\mu_T)\in \mathscr{M}_+([0,T] \times X \times U)\times \mathscr{M}_+(K)$ of Problem~\eqref{opt:relaxed}.
\end{definition}

From here, we will make the following assumption.
\begin{assumption}\label{FEASIBILITY}Assume that:
\begin{itemize}
\item Problem~\eqref{optimalcontrol} is feasible;
\item $X$, $K$ and $U$ are compact and $T < \infty$, which implies that $\mathcal{M}$ is non-empty. 
\end{itemize}
\end{assumption}

A simple observation shows that $\rho^\star \le \rho$. Whenever $\rho^\star < \rho$, we say that there is a \emph{relaxation gap} between the original Problem~\eqref{optimalcontrol} and the relaxed Problem~\eqref{opt:relaxed}. A fundamental question is to understand under what assumptions there is no relaxation gap. A classical set of assumptions from~\cite{vinter1993} is:

\begin{assumption}\label{ass:nogap}
The following conditions hold:
\begin{itemize}
    \item $\forall (t,x)\in [0,T]\times X$, \ \text{the function} \ $v\mapsto \inf_{u\in U} \{h(t,x,u) \mid v=f(t,x,u)\}$ \ \text{is convex};
    \item  $\forall x\in X$, the set $f(x,U)$ is convex;
    \item $h$ is lower semicontinuous and $f$ is Lipschitz;
\end{itemize}
\end{assumption}

We introduce the following definition of occupation measures, which correspond to the solutions of the measure Problem \eqref{opt:relaxed} which provide solutions of the original Problem~\eqref{optimalcontrol}.

\begin{remark}Note that, under Assumption~\ref{ass:nogap}, the occupation measures can be defined as measures associated with classical solutions of Equation~\eqref{general}, instead of considering solutions in the sense of Young measures (for more information, see for instance~\cite[Proof of Theorem 2.3]{vinter1993}).
\end{remark}

\subsection{Moment sequences, relaxation scheme, invariance}
For positive integers $k,N$, define $\N_k^N=\{\omega\in \N^N \mid \sum_{j=1}^N \omega_{j} \leq k\}$.
For positive integers $n,m$, let us define a multi-index $\alpha$ as a vector of $n$ nonnegative integers $\alpha=(\alpha_1,\dots,\alpha_n)$, and a multi-index $\beta$ as a vector of $m$ nonnegative integers $\beta=(\beta_1,\dots,\beta_m)$, and $s\geq 1$.
For the sake of notation, a polynomial $q\in \R[t,x,u]$ (respectively, $q\in \R[x]$)  is written as \[q(t,x,u)=\sum_{\gamma = (s, \alpha,\beta)} q_{\gamma} t^s x_{1}^{\alpha_1}\dots x_{n}^{\alpha_n}u_1^{\beta_1}\dots u_m^{\beta_m},\] (respectively, $q(x)=\sum_{\alpha} q_{\alpha} x_{1}^{\alpha_1}\dots x_{n}^{\alpha_n}$, where $\alpha\in \N^{n}$).

\begin{definition}\label{MOM_MAT}
\begin{itemize}
\item Given a real sequence $z=(z_\gamma)_{\gamma\in \N^{1+n+m}}$ (respectively, $y=(y_\alpha)_{\alpha\in \N^{n}}$), define the so-called Riesz linear functional \[L_z(q):=\sum_\gamma q_\gamma z_\gamma,\] for every $q\in \R[t,x,u]$ (respectively, $L_y(q):=\sum_\alpha q_\alpha y_\alpha$, for every $q\in \R[x]$).
\item For each $k \in \N$, define the \emph{moment matrix} $M_k(z)$ by \[(M_k(z))_{\beta,\gamma}:=z_{\beta+\gamma},\] for every $\beta,\gamma \in \N_{k}^{1+n+m}$ (respectively, $(M_k(y))_{\beta,\gamma}:=y_{\beta+\gamma}$, for every $\beta,\gamma \in \N_{k}^{n}$). 
\item For each $k \in \N$ and $q\in \R[t,x,u]$ (respectively, $q\in \R[x]$) , define the \emph{localizing matrix} as \[\left(M_k(q z)\right)_{\beta,\gamma}:=\sum_{\delta\in \N^{1+n+m}}  q_\delta z_{\delta+\beta+\gamma},\] for every $\beta,\gamma\in \N_{k}^{1+n+m}$ (respectively, $\left(M_k(q y)\right)_{\beta,\gamma}:=\sum_{\delta\in \N^{n}}  q_\delta y_{\delta}$, for every $\beta,\gamma\in \N_{k}^{n}$).
\item For the sake of notation, define $z(t):=(z_{j,0,0})_{j\in \N}$, $z(x):=(z_{0,\alpha,0})_{\alpha\in \N^n}$, and $z(u):=(z_{0,0,\beta})_{\beta\in \N^m}$.
\end{itemize}
\end{definition}

\begin{definition}\label{G_INV_SEQ}
A pseudo-moment sequence $z=(z_\alpha)_{\alpha\in \mathcal{B}}$ (respectively, $y=(y_\alpha)_{\alpha\in \tilde{\mathcal{B}}}$)  is defined as the image of a monomial basis $\mathcal{B}$ (respectively, $\tilde{\mathcal{B}}$) of $\R[t,x,u]$ (respectively, $\R[x]$) by a linear functional $L:\R[t,x,u]\to \R$ (respectively, $L:\R[x]\to \R$).
\end{definition}

\subsubsection{Dense relaxation~\label{LET_DENSE}}
\begin{assumption}\label{ass:final_time}
Assume that the final time $T>0$ is fixed and without loss of generality, let us normalize it to $T=1$.
\end{assumption}

\begin{remark}
Assumption~\ref{ass:final_time} can be removed by considering the final time $T$ as a supplementary variable, as it was done in~\cite{OCP08}. All the results of the present paper remain valid in this case provided that a finite bound is imposed on the final time. 
\end{remark}

Consider the polynomials $(v_j)_j$, $(\theta_j)_j$, $(w_j)_j$ as defined in Section~\ref{FRAME}, and define $d(X,K,U)=\max_{j,l,k} \left(\text{deg}(v_j),\text{deg}(\theta_l),\text{deg}(w_k)\right)$.
Following the steps of~\cite{OCP08}, we propose the following dense relaxation hierarchy, 
indexed by $k\geq k_0$, where $k_0 := \lceil \text{max}\left(\text{deg}(f),\text{deg}(h),\text{deg}(H),d(X,K,U)\right) /2 \rceil$, for two pseudo moment sequences $z=(z_\alpha)_{\alpha\in \mathcal{B}}$ and  $y=(y_\alpha)_{\alpha\in \tilde{\mathcal{B}}}$, where $\mathcal{B}$ (respectively, $\tilde{\mathcal{B}}$) is a monomial basis of $\R[t,x,u]$ (respectively, $\R[x]$).
It is denoted by~\eqref{sdp_dense}, and defined as follows:

\begin{equation} \tag{$Q_k$}\label{sdp_dense}
\begin{aligned}
\rho_k := \inf\limits_{y,z} \quad & L_z(h)+L_y(H)\\
\textrm{s.t.} \quad &   M_k(y),M_k(z)\succeq 0\\
   & M_{k-\lceil\text{deg}(v_j)/2\rceil}(v_j z(x))\succeq 0  \\
   &M_{k-\lceil\text{deg}(w_j)/2\rceil}(w_j z(u))\succeq 0 \\
  &M_{k-\lceil\text{deg}(\theta_j)/2\rceil}(\theta_j y)\succeq 0 \\
  &M_{k-1}(t(1-t) z(t))\succeq 0 \\
  & L_y(\phi)-L_z\left(\partial \phi/ \partial t+\langle \nabla_x \phi, f \rangle \right)=\phi(0,x_0), \\ 
  & \forall \phi=(t^sx^\alpha) \in \R[t,x] \ \text{s.t.} \ s+|\alpha|\leq 2k+1-\text{deg}(f),\\
\end{aligned}
\end{equation}
where $z(t)$, $z(x)$ and $z(u)$ are defined as in Definition~\ref{MOM_MAT}.
Note that last equality is a linear equality constraint involving the pseudo-moment variables $y$ and $z$ holds if and only if Equation~\eqref{Liouville} is satisfied for every $\phi\in \R[t,x] \ \text{s.t.} \ \text{deg}(\phi)\leq 2k+1-\text{deg}(f)$. 

\subsubsection{Putinar's Positivstellensatz}
The key argument in the convergence of the hierarchy is Putinar's Positivstellensatz, which is stated in its simplest form as follows.
Let $\Omega\subset \R^N$ be a compact semi-algebraic set defined as \[\Omega=\{x\in \R^N \ \mid \ g_j(x)\geq 0, \ j\in \{1,\dots,m\}\},\] for some family of polynomials $(g_j)_{j\in \{1,\dots,m\}}$ in $\R[x]$.

\begin{assumption}\label{ass:condition}
There exists a polynomial $q\in \R[x]$ such that $q(x)\geq 0$ is compact and $q(x)=q_0(x)+\sum_{j=1}^m g_j(x) q_j(x)$, where the polynomials $(q_j)_{j\in \{1,\dots,m\}}$ are sums of squares (SOS) polynomials in $\R[x]$.
\end{assumption}
\begin{thm}[Putinar Positivstellensatz~\cite{PUTINAR}]\label{PUTINAR_THM}
Assume that Assumption~\ref{ass:condition} holds.
Then we have the following:
\begin{itemize}
    \item If $f\in \R[x]$ and $f>0$ on $\Omega$, then $f=f_0+\sum_{j=1}^m f_j g_j$, for some family $(f_j)_{j\in \{1,\dots,m\}}$ of sums of squares (SOS) polynomials in $\R[x]$.
    \item Let $y=(y_{\alpha})_{\alpha\in \N^N}$ be a sequence of real numbers. If for every $j\in \{1,\dots,m\}$ and $k\in \N$, $M_k(y)\succeq 0$ and $M_k(g_j y)\succeq 0$, then $y$ has a representing measure with support contained in $\Omega$, i.e., there exists $\mu\in \mathscr{M}_+(\R^N)$ with $\text{supp}(\mu)\subseteq \Omega$ such that $y_\alpha=\int_{\R^N} x^\alpha d\mu(x)$, for every $\alpha\in \N^N$.
\end{itemize}

\end{thm}

\subsubsection{Convergence result}
Under Assumption~\ref{FEASIBILITY}, \ref{ass:nogap} and~\ref{ass:condition}, the results of~\cite[Theorem 3.6]{OCP08} and~\cite[Theorem 5]{TACCHI22} allow us to prove the following.
\begin{prop}\label{CONV_MOMENTUM}
Assume that Assumption~\ref{ass:nogap} holds for $\Omega=[0,T]\times X \times U$ and $\Omega=K$, and that Assumption~\ref{ass:condition} holds. Then we have the following:
\begin{itemize}
    \item the sequence $(\rho_k)_k$ is nondecreasing and $\rho_k \to \rho^\star$ when $k\to +\infty$
    \item if Problem~\eqref{opt:relaxed} has a unique solution $(\mu,\mu_T)$, then any sequences $(z_{s \alpha \beta}^k)_{k\geq k_0}$ and $(y_{\alpha}^k)_{k\geq k_0}$ optimal in~\eqref{sdp_dense} satisfy for every $s\in \N$, $\alpha\in \N^n$ and  $\beta\in\N^m$, \[\lim_{k\to\infty} z_{s\alpha\beta}^k = \int_{[0,T]\times X \times U} t^{s}x^{\alpha}u^{\beta} d\mu(t,x,u),\] and \[\lim_{k\to\infty} y_{\alpha}^k = \int_{K} x^{\alpha} d\mu_T(x).\]
\end{itemize}
\end{prop}

 \subsubsection{Structure of the set of optimal measures}
In the case where the set of optimal trajectories of Problem~(\ref{optimalcontrol}) is not reduced to a single trajectory, the set of optimal measures for Problem~(\ref{opt:relaxed}) is not reduced to occupation measures associated with trajectories in the sense of Definition~\ref{OCCUP_MEAS}. However, as shown in~\cite{vinter1993}, there are strong structural results for the solutions of Equation~\eqref{Liouville}, which will be crucial for the methods exposed in the present paper.
Let us first recall a basic notion.
\begin{definition} 
An extreme point of a convex set $A$ is a point $x\in A$ with the property that if $x=\theta y+(1-\theta)z$ with $y,z\in A$ and $\theta \in (0,1)$, then $x = y =z$.
\end{definition}

By~\cite[Theorem 1.3]{vinter1993}, we have the following  results for the solutions of Equation~\eqref{Liouville}.
\begin{proposition}\label{EXTREME}
The set of solutions of Equation~\eqref{Liouville}  satisfying $(\mu,\mu_T)\in \mathscr{M}_+([0,T] \times X \times U)\times \mathscr{M}_+(K)$ 
is a convex compact set in the weak-$\star$ topology. Its set of extreme points is non empty and can be identified as occupation measures, as in Definition~\ref{OCCUP_MEAS}. 
\end{proposition}
Using~\cite[Corollary 1.4]{vinter1993}, we have the following property.
\begin{corollary}
For every $(\mu,\mu_T)$ solution of Equation~\eqref{Liouville}, there exist probability measures $\nu$ on $S$ and $\nu_T$ on $S_T$ satisfying \[\int_{[0,T] \times X \times U} \phi(t,x,u) d\mu(t,x,u)= \int_{S} \left(\int_{[0,T] \times X \times U} \phi(t,x,u) d\gamma(t,x,u)\right) d\nu(\gamma),\] \[\int_{K} \varphi(x) d\mu_T(x)= \int_{S_T} \left(\int_{K} \varphi(x) d\gamma(x)\right) d\nu_T(\gamma),\] for all $(\phi,\varphi)\in C^1([0,T] \times X \times U)\times C^1(K)$.
\end{corollary}

 By linearity and continuity of the cost functional of Problem~\eqref{opt:relaxed} w.r.t. $(\mu,\mu_T)\in \mathscr{M}_+([0,T] \times X \times U)\times \mathscr{M}_+(K)$, we obtain the following result.
\begin{corollary}\label{EXTR_POINT}
The set $\mathcal{M}$ is a convex compact set of $\mathscr{M}_+([0,T] \times X \times U)\times \mathscr{M}_+(K)$ in the weak-$\star$ topology. Its set of extreme points $\mathcal{E}=S\times S_T\subset \mathscr{M}_+([0,T] \times X \times U)\times \mathscr{M}_+(K)$ is non empty and can be identified as occupation measures, in the sense of Definition~\ref{OCCUP_MEAS}.
\end{corollary}

\section{Techniques and properties for symmetry reduction~\label{tech}}

The idea behind the algorithms which will be exposed in further Sections~\ref{RECOV} and \ref{SELECTO} is to reduce the problem via a symmetry decomposition of the solution set $\mathcal{M}$, by the resolution of appropriate $G$-invariant LP problems on the set of $G$-invariant non-negative measures. It requires an understanding of the group action on the space of optimal trajectories $\mathcal{T}$, on polynomials, as well as on the space of non-negative measures, that will be described in this section.
 We will first state two crucial results, namely Lemma~\ref{Liouville_invariant} and Proposition~\ref{INVARIANCE} about the structure of the set of solutions to~\eqref{opt:relaxed}. Then we will introduce symmetry-reduced moments matrices, which will be the key tool enabling us to solve approximately the symmetry-reduced measure LPs via SDPs.

In what follows, we make the following assumption.
\begin{assumption}
    The group $G$ is finite, and Problem~\eqref{optimalcontrol} is $G$-invariant. 
\end{assumption}

\subsection{Group action on measures and properties of solutions of the linear problem on measures}

\begin{definition}\label{MEAS_DEFI}
\begin{itemize}
\item For a measure $\mu\in \mathscr{M}_+([0,T] \times X \times U)$, define the pushforward measure $g_{\#}\mu$ of $\mu$, defined for every $g\in G$, for every borelian subset $A$ of $[0,T]$, borelian subset $B$ of $X$, and borelian subset $C$ of $U$, by $g_{\#}\mu(A\times B\times C)=\mu(A\times g^{-1}(B)\times \tau(g)^{-1}(C))$.
Define similarly, for a borelian subset $D$ of $K$, $g_{\#}\mu_T(D)=\mu_T(g^{-1}(D))$ for $\mu_T\in \mathscr{M}_+(K)$.
\item Define for every $\mu\in \mathscr{M}_+([0,T] \times X \times U)$, the Reynold operator $R$ by $R(\mu)=\frac{1}{|G|}\sum_{g\in G} g_{\#}\mu$.
Define similarly $R_T(\mu_T)=\frac{1}{|G|}\sum_{g\in G} g_{\#}\mu_T$, for $\mu_T\in \mathscr{M}_+(K)$.
\item We say that a measure $\mu\in \mathscr{M}_+([0,T] \times X \times U)$ (respectively, $\mu_T\in \mathscr{M}_+(K)$) is $G$-invariant if $g_{\#}\mu=\mu$ (respectively, $g_{\#}\mu_T=\mu_T$), for every $g\in G$. Denote the set of $G$-invariant pairs $(\mu,\mu_T)\in \mathscr{M}_+([0,T] \times X \times U)\times \mathscr{M}_+(K)$ of measures which are solution of Problem~\eqref{opt:relaxed} by $\mathcal{M}^G$.

\item For every $G$-invariant $(\mu^\star,\mu_T^\star) \in \mathscr{M}_+([0,T] \times X \times U)\times \mathscr{M}_+(K)$, define the (possibly infinite dimensional) convex set \[\mathcal{A}_{\mu^\star,\mu_T^\star}=\{(\mu,\mu_T)\in \mathscr{M}_+([0,T]\times X \times U)\times \mathscr{M}_+(K) \ \mid  R(\mu)=\mu^\star, \ R_T(\mu_T)=\mu_T^\star \}.\] 
\end{itemize}
\end{definition}

A first observation is that if the control system and initial condition $x_0$ and final target set $K$ are invariant by $G$, then the Liouville Equation~\eqref{Liouville} is invariant by $G$, i.e., the set $\mathcal{M}$ is $G$-invariant, as stated in next Lemma.

\begin{lemma}\label{Liouville_invariant}
Assume that the initial condition $x_0\in X$ and final target set $K\subseteq X$ are invariant by $G$. Then Equation~\eqref{Liouville} is invariant by $G$, i.e., the pair $\left(\mu, \mu_T\right) \ \text{satisfy~\eqref{Liouville}}$ if and only if $\left(g_{\#}\mu, g_{\#}\mu_T\right) \ \text{satisfy~\eqref{Liouville}}$, for every $g\in G$.
\begin{proof}
Consider $\mu$ satisfying~\eqref{Liouville} and $\phi\in C^1([0,T]\times X)$. For every $g\in G$, we have
\begin{align*}
\int_K \phi(T,g(x))d\mu_T(x)-\phi(0,g(x_0))
=& \int_{[0,T]\times X \times U} \left( \frac{\partial \phi}{\partial t}(t,g(x))+\langle \nabla_x \phi(t,g(x)), g(f(x,u))\rangle \right) d\mu(t,x,u)\\
=& \int_{[0,T]\times X \times U} \left( \frac{\partial \phi}{\partial t}(t,g(x))+\langle \nabla_x \phi(t,g(x)), f(g(x),\tau(g)(u))\rangle \right) d\mu(t,x,u)\\
=& \int_{[0,T]\times X \times U} \left( \frac{\partial \phi}{\partial t}(t,x)+\langle \nabla_x \phi(t,x), f(x,u)\rangle \right) dg_{\#}\mu(t,x,u),
\end{align*}
where the last equality is obtained by symmetry invariance of $f$.
We deduce the result using the equality $g(x_0)=x_0$ and $\int_K \phi(T,g(x))d\mu_T(x)=\int_K \phi(T,x)dg_{\#}\mu_T(x)$ by $G$-invariance of $x_0$ and $K$. 
\end{proof}
\end{lemma}



We state the following structural results for the set of $G$-invariant solutions of Problem~\eqref{opt:relaxed}, whose proof is postponed to Appendix~\ref{PROOF_INVARIANCE}.
\begin{proposition}\label{INVARIANCE}
  
The set $\mathcal{M}^G\subseteq \mathcal{M}$ defined in Definition~\ref{MEAS_DEFI} satisfies the following properties:
\begin{itemize}
    \item 
    We have the equality \[\mathcal{M}^G=\{\left(R(\mu),R_T(\mu_T)\right)\mid \ (\mu,\mu_T)\in \mathcal{M}\},\] and $\mathcal{M}^G$ is a weak-$\star$ compact convex set.
    \item The extreme points of $\mathcal{M}^G$ are the images by the Reynold operator of the occupation measures optimal in~\eqref{opt:relaxed}, i.e., \[\mathcal{E}_G=\{(R(\mu),R_T(\mu_T)) \mid \ (\mu,\mu_T)\in \mathcal{M} \ \text{is a pair of occupation measures}\}.\]
\item We have $\mathcal{M}^G=\overline{\text{Conv}(\mathcal{E}_G)}$.
\end{itemize}
\end{proposition}

 Note that the results of this paper are not specific to the case where the set of optimal trajectories $\mathcal{T}$ for Problem~\eqref{optimalcontrol} is finite. However, in the latter case, we have the following simple features, which are illustrated on the simplified finite dimensional representation of Figure~\ref{FIBRATION}.
 

\begin{figure}[!ht]
\begin{center}
        \includegraphics[scale=0.8]{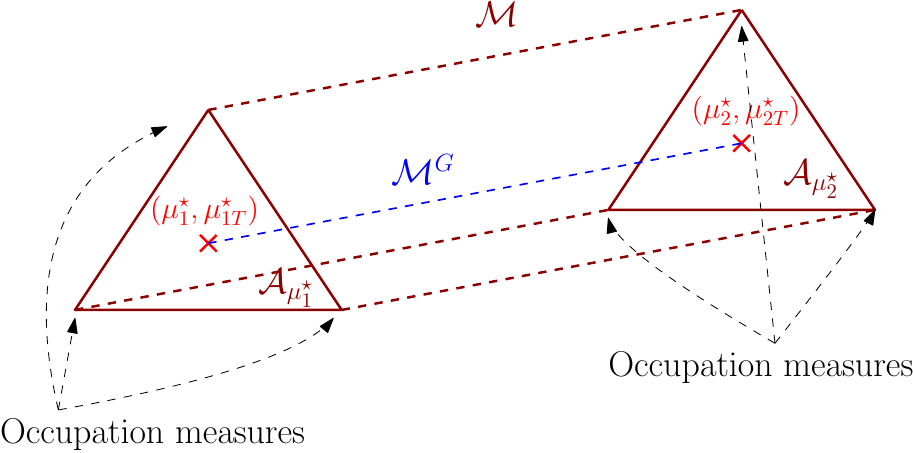}
        \caption{
        Simplified representation of the set of optimal measures $\mathcal{M}$ for Problem~\eqref{opt:relaxed}. The occupation measures are the extreme points of $\mathcal{M}$, and $(\mu_1^\star,\mu_{1T}^\star)$, $(\mu_2^\star,\mu_{2T}^\star)$ are the extreme points of $\mathcal{M}^G$.}
        \label{FIBRATION}
\end{center}
\end{figure}

\begin{remark}[Finite dimensional case]\label{FINITE_DIM}
When there is a finite number $k\geq 1$ of optimal trajectories, the set of optimal measures is \[\mathcal{M}=\{(\mu,\mu_T) \in \mathscr{M}_+([0,T]\times X \times U)\times \mathscr{M}_+(K) \mid \mu=\sum_{i=1}^k \lambda_i \nu^i, \mu_T=\sum_{i=1}^k \lambda_i \nu_T^i \ \sum_{i=1}^k \lambda_i=1, \ \lambda_i\geq 0\},\]
where $(\nu^i,\nu_T^i)_{i\in \{1,\dots,k\}}$ are pairs of occupation measures associated with the optimal trajectories. In this case, the group action induces a permutation of the extreme points of $\mathcal{M}$.
\end{remark}

\subsection{Group action on pseudo-moment sequences and symmetry-adapted moment matrices~\label{SYM_MOM}}
Here we give some important properties concerning the numerical resolution of the SDPs. In particular, using the classical isotopic decomposition of $G$-modules, we define the symmetry reduced moment matrices, as it was done in~\cite{SYMSDP}.
It is first presented for general polynomials depending on a variable $x \in \R^n$; its adaptation to the variables $(t,x,u)$ involved in optimal control problems is straightforward.
It requires standard tools of representation theory, which can be found, for instance in~\cite{serre1977linear}.
Consider a group action of a finite group $G$ of $GL_m(\R)$ on a subspace $X\subseteq \R^n$.
\begin{definition}\label{INVARIANT_SEQUENCE}
\begin{itemize}
\item A linear map $L:\R[x]\to \R$ is $G$-invariant if $L(P(g(x)))=L(P(x))$ for every $g\in G$ and $P\in \R[x]$.
\item A pseudo-moment sequence $y=(y_\alpha)_{\alpha\in \mathcal{B}}$ associated with a linear map $L:\R[x]\to \R$ and a monomial basis $\mathcal{B}$ of $\R[x]$ is $G$-invariant if $L$ is $G$-invariant.
\item Define the Reynold operators on linear maps $L:\R[x] \to \R$, by
$R(L)=\frac{1}{|G|} \sum_{g\in G} L^g$, where $L^g(P)(x)=L(P(g(x)))$, for every $g\in G$ and $P\in \R[x]$. 
\item For a polynomial $P\in \R[x]$, define the polynomial $\mathcal{R}(P)=\frac{1}{|G|} \sum_{g\in G} P^g$.
\end{itemize}
\end{definition}

As described in~\cite[Section 3]{SYMSDP}, the set $\R[x]$ can be seen as a real $G$-module defined by a group action defined for a polynomial $P\in \R[x]$, define $P^g(x)=P(g(x))$ for every $g\in G$. As a direct consequence of the so-called \emph{isotopic decomposition} of $G$-modules, the set $\R[x]$ can be decomposed as a sum of irreducible $G$-modules, as 
\begin{equation}\label{IRRED_DEC}
\R[x]\otimes \C=\bigoplus_{l=1}^q V_l,
\end{equation}
where 
\[V_l=\bigoplus_{j\in J_l} W_{lj}\] where for fixed $l\in \{1,\dots,q\}$, the sets $W_{lk}$ are pairwise isomorphic complex irreducible $G$-modules, in the sense that they do not admit non-trivial $G$-invariant subspaces.

Now we can give the following definition of symmetry-reduced moments matrices, which are defined according to this decomposition, by
considering a basis $(s_{j,v}^l)_v$ of $W_{lj}$, setting $\mathcal{S}^l=\{s_{j,1}^l, \ j\in J_l\}$, and defining the truncated sets $\mathcal{S}^l_k=\{(s_{1}^l,\dots,s_{\eta_l}^l)\}\subseteq \mathcal{S}^l$ of the basis elements of $\mathcal{S}^l$ of degree at most $k$. Note that, as claimed in~\cite[Section 3]{SYMSDP}, the set $\mathcal{S}^l\subseteq \C[x]$ can be assumed to be real.

\begin{definition}[Symmetry-reduced moment and localizing matrices]\label{RED_MATRIX}
Define the symmetry-reduced moments matrix, defined for a $G$-invariant pseudo-moment sequence $y=(y_\alpha)_{\alpha\in \mathcal{B}}$ associated with a $G$-invariant linear map $L$, as \[M_k^G(y)=\bigoplus_{l=1}^k M_{kl}^G(y),\] where the $(u,v)$ entry of  $M_{kl}^G(y)$ is equal to $L(s_u^l s_v^l)$.
\end{definition}

\begin{remark}
\begin{itemize}
\item For the study of Problem~\eqref{optimalcontrol}, one needs to consider the variables $(t,x,u)$ and the action of an element $g\in G$ as $(t,x,u)\mapsto (t,g(x),\tau(g)(u))$, as described in Section~\ref{FRAME}. It is then necessary to introduce both $G$-invariant pseudo-moment sequence w.r.t. the variables $(t,x,u)$ and $x$, that we will denoted, respectively as $z=(z_\gamma)_{\gamma\in \mathcal{B}}$ and $y=(y_\alpha)_{\alpha\in \tilde{\mathcal{B}}}$, for a monomial basis $\mathcal{B}$ (respectively, $\tilde{\mathcal{B}}$) of $\R[t,x,u]$ (respectively, $\R[x]$).
\item The same symmetry reduction holds for the partial moment matrices associated with $z(t)$, $z(u)$ and $z(x)$, as well as for the localizing matrices, as defined in~Definition~\ref{MOM_MAT}.
\end{itemize}
\end{remark}

To this purpose, we introduce the following set.
\begin{definition}\label{INV_SEK}
Let $\mathcal{K}^G$ be the set of pairs $(y,z)$ of $G$-invariant sequences written as $y=(y_\alpha)_{\alpha\in \tilde{\mathcal{B}}}$ and $z=(z_\gamma)_{\gamma\in \mathcal{B}}$, for some monomial basis $\mathcal{B}$ of $\R[x]$ and $\tilde{\mathcal{B}}$ of $\R[t,x,u]$.
\end{definition}


\subsubsection{Sign-symmetries}
In the numerical implementations proposed in Section~\ref{APPLIQ}, we will consider the practical case of  sign symmetries for the optimal control problem, whose algebraic structure is described below.
 \begin{definition}\label{LACDES}
We say that Problem~\eqref{optimalcontrol} has a sign-symmetry if it is $G$-invariant in the sense of Definition~\ref{DEF_SYM}, with $G$ is generated by a finite set of $n\times n$ diagonal matrices $(D_j)_{j\in\{1,\dots,s\}}$, $s\geq 1$, with coefficients equal to $\pm 1$, such that, for every $j\in \{1,\dots,s\}$, $\tau(D_j)$ is a $m\times m$ diagonal matrix with coefficients equal to $\pm 1$. 
 \end{definition}


In the case where Problem~\eqref{optimalcontrol} has a sign-symmetry generated by a single $n\times n$ diagonal matrix $D$ we use the following decomposition and reduction of moment matrices:
\begin{itemize}
\item  We consider the truncations of the irreducible decomposition~\eqref{IRRED_DEC} to degree $k$ as \[\R[t,x,u]_{k}=V_k^1\bigoplus V_k^2,\] where \[V_k^1=\{P\in \R[t,x,u]_{k} \mid P(t,D(x),\tau(D)(u))=P(t,x,u)\},\] and \[V_k^2=\{P\in \R[t,x,u]_{k} \mid P(t,D(x),\tau(D)(u))=-P(t,x,u)\}.\]
Similarly consider \[\R[x]_{k}=\tilde{V}_k^1\bigoplus \tilde{V}_k^2,\] where \[\tilde{V}_k^1=\{P\in \R[x]_{k} \mid P(D(x))=P(x)\},\] and \[\tilde{V}_k^2=\{P\in \R[x]_{k} \mid P(D(x))=-P(x)\}.\] 
\item We generate $G$-invariant pseudo-moment sequences $(y,z)\in \mathcal{K}^G$ as sequences indexed w.r.t. the products of elements of monomial basis $\mathcal{S}_k^j$ associated with $V_k^j$ (respectively, $\tilde{\mathcal{S}}_k^j$ associated with $\tilde{V}_k^j$), for $j\in \{1,2\}$.
\item 
As defined in Definition~\ref{RED_MATRIX}, the moment matrix
$M_k^G(z)$ can be decomposed as two blocks \[M_k^G(z)=M_{k1}^G(z)\bigoplus M_{k2}^G(z),\] where for $l\in \{1,2\}$, the square matrix $M_{kl}^G(z)$ has dimension equal to $\text{dim} (V_k^l)$.
\item The matrix $M_k^G(y)$ can be decomposed as two blocks
$M_k^G(y)=M_{k1}^G(y)\bigoplus M_{k2}^G(y)$, where for $l\in \{1,2\}$, the square matrix $M_{kl}^G(y)$ has dimension equal to $\text{dim}  (\tilde{V}_k^l)$.
\end{itemize}

\begin{example}\label{EXX}
For the case $n=m=1$, $D=\tau(D)=-1$, $k=2$, in the concatenation of the basis $\mathcal{S}_2^1=(1,t,t^2,x^2,xu,u^2)$ and  $\mathcal{S}_2^2=(x,u,tx,tu)$ and associated pseudo-moment sequences $z_{\mathcal{S}_2^1}$ indexed w.r.t. products of the elements of $\mathcal{S}_2^1$  and $z_{\mathcal{S}_2^2}$ indexed w.r.t. products of the elements of $\mathcal{S}_2^2$,  the moment matrix reads as two blocks \[M_2^G(z)=\begin{pmatrix}
M_2(z_{\tilde{\mathcal{S}}_2^1})&0\\
0&M_2(z_{\tilde{\mathcal{S}}_2^2})
\end{pmatrix},
\]
of size $6\times 6$ and $4\times 4$, instead of a squared $10\times 10$ matrix in the dense case.
Similarly, in the concatenation of the basis $\tilde{\mathcal{S}}_2^1=(1,x^2)$ and  $\tilde{\mathcal{S}}_2^2=(x)$, and associated pseudo-moment sequences $y_{\tilde{\mathcal{S}}_2^1}$ indexed w.r.t. the elements of $\tilde{\mathcal{S}}_2^1$  and $y_{\tilde{\mathcal{S}}_2^2}$ indexed w.r.t. the elements of $\tilde{\mathcal{S}}_2^2$,,  the moment matrix reads as two blocks \[M_2^G(y)=\begin{pmatrix}
M_2(y_{\tilde{\mathcal{S}}_2^1})&0\\
0&M_2(y_{\tilde{\mathcal{S}}_2^2})
\end{pmatrix}
\] of size $2\times 2$ and $1\times 1$.
\end{example}

\begin{remark}\label{sev}
In the case where $G$ is generated by several diagonal matrices in the sense of Definition~\ref{LACDES}, the moment matrices can be reduced into more numerous blocks, in the same way as what is done for the dual version on positive polynomials in~\cite[Section~III-C]{LOF}.
\end{remark}


\section{Symmetry adapted relaxation for optimal control \label{HYERES_ARCHI}}

Now we adapt the moment-SOS relaxation exposed in~\cite{OCP08} by adapting the polynomial optimization results of~\cite{SYMSDP} to the setting of Problem~\eqref{opt:relaxed}. 

We restrict Problem~\eqref{opt:relaxed} to the optimization over $G$-invariant measures.
\begin{equation} \label{opt:relaxed_bis}\tag{$\mathcal{O}^G$}
\begin{aligned}
\rho^G:= \inf\limits_{\mu,\mu_T} \quad & \langle \mu,h\rangle\\
\textrm{s.t.} \quad &   (\mu,\mu_T) \text{\;\,satisfy (\ref{Liouville})}\\
   & \text{supp}(\mu_T) \subseteq K  \\
  &(\mu,\mu_T) \in \mathscr{M}_+([0,T]\times X\times U) \times \mathscr{M}_+(K)   \\
  &\mu, \mu_T \ \text{are} \ G-\text{invariant}, \ \text{i.e.} \ g_{\#}\mu=\mu, \ g_{\#}\mu_T=\mu_T, \ \text{for every} \ g\in G.
\end{aligned}
\end{equation}

As Problem~\eqref{opt:relaxed} is feasible, we can consider $\mu,\mu_T$ in its feasible set, and obtain directly by Lemma~\ref{Liouville_invariant} and the affine dependence of Equation~\eqref{Liouville} in the pair $(\mu,\mu_T)$ that $\left(\mu^\star,\mu_T^\star\right)=\left(R(\mu),R_T(\mu_T)\right)$ belongs to the feasible set of Problem~\eqref{opt:relaxed_bis}. It follows that Problem~\eqref{opt:relaxed_bis} is feasible, and the corresponding cost is unchanged by the transformation via the transformation $(\mu,\mu_T)\mapsto (R(\mu),R_T(\mu_T))$, so that $\rho^G=\rho^\star$, where $\rho^\star$ is the optimal cost of Problem~\eqref{opt:relaxed}. As a consequence, the set of optimal measures for Problem~\eqref{opt:relaxed_bis} is exactly the set $\mathcal{M}^G$ defined in Definition~\ref{MEAS_DEFI}.

\begin{remark}\label{NO_PUSH}
Note that Problem~\eqref{opt:relaxed_bis} does not correspond to a measure relaxation of an optimal control problem in the form described by~\eqref{optimalcontrol}.
Contrarily to what one could think at first sight and to what has been done in different settings in~\cite{GG99,HUBERT,OHSAWA_SYMMETRIES}, we do not map the trajectories of~\eqref{optimalcontrol} onto trajectories of a reduced control system via a change of variables. 
Instead, we take advantage of the linearity of the measure relaxation formulations~\eqref{opt:relaxed} and \eqref{opt:relaxed_bis}.
\end{remark}



In order to propose a suitable semi-definite approximation of Problem~\eqref{opt:relaxed_bis} with convergence guarantees, we propose to follow the steps of~\cite[Section 3]{SYMSDP} via the use of decompositions described in Section~\ref{tech} allowing one to decompose the moment matrices and localizing matrices into several low dimensional blocks, as defined in Definition~\ref{RED_MATRIX}. 

For $(y,z)\in \mathcal{K}^G$, denote the symmetry reduced moment matrices by $M_k^G(y)$ and $M_k^G(z)$ defined for $k\geq 1$, as in Definition~\ref{RED_MATRIX}. 
Following the steps of~\cite{SYMSDP}, we propose a hierarchy indexed by $k\geq k_0$, where $k_0$ is chosen as in Section~\ref{LET_DENSE}, denoted by~\eqref{sdp_sym_2}:
\begin{equation} \tag{$Q_k^G$}\label{sdp_sym_2}
\begin{aligned}
\rho_k^G := \inf\limits_{(y,z)\in \mathcal{K}^G} \quad & L_z(h)+L_y(H)\\
\textrm{s.t.} \quad &   M_k^G(y),M_k^G(z)\succeq 0\\
   & M_{k-\lceil\text{deg}(v_j)/2\rceil}^G(v_j z(x))\succeq 0  \\
   &M_{k-\lceil\text{deg}(w_j)/2\rceil}^G(w_j z(u))\succeq 0 \\
  &M_{k-\lceil\text{deg}(\theta_j)/2\rceil}^G(\theta_j y)\succeq 0 \\
  &M_{k-1}^G(t(1-t) z(t))\succeq 0 \\
  & L_y(\phi)-L_z\left(\partial \phi/ \partial t+\langle \nabla_x \phi, f \rangle \right)=\phi(0,x_0), \ \forall \phi=(t^s x^\alpha) \in \R[t,x] \ \text{s.t.} \ s+|\alpha|\leq 2k+1-\text{deg}(f)\\
\end{aligned}
\end{equation}

\begin{remark}
In accordance with Definition~\ref{INVARIANT_SEQUENCE}, $G$-invariant pseudo-moment sequences $(y,z)\in \mathcal{K}^G$ can be generated as $R(L)(\mathcal{B})$, where $\mathcal{B}$ is a monomial basis of $\R[x]$ and $L:\R[x]\to \R$ is a linear map. The same holds for $z=(z_{\gamma})_{\gamma}$, replacing $\R[x]$ by $\R[t,x,u]$.
Up to an elimination of redundant terms, those sequences can be reindexed w.r.t. a basis of invariant polynomials of $\R[x]$ and $\R[t,x,u]$, the latter being further introduced in Defintion~\ref{INVARIANT_POLY}.
\end{remark}

\begin{thm}
Assume that Assumption~\ref{ass:nogap} and~\ref{ass:condition} hold.
Then the sequence $(\rho_k^G)$ is non-decreasing and converges to $\rho^*$ when $k\to +\infty$.
\begin{proof}
We adapt the arguments of~\cite[Theorem 3.7]{SYMSDP}. 
Let $y,z$ be two pseudo-moment sequences associated with linear maps $L:\R[t,x,u] \to \R$ and $\bar{L}: \R[x] \to \R$, which are solution of~\eqref{sdp_dense}, which is feasible by~\cite[Theorem 3.6]{OCP08}.
Then the two $G$-invariant sequences $\tilde{y},\tilde{z}$ associated with $R(L)$ and $R_T(\bar{L})$, where $R(L)$ and $R_T(\bar{L})$ are defined as in Definition~\ref{INVARIANT_SEQUENCE}, are solution of~\eqref{sdp_sym_2}, so that we have that the SDP~\eqref{sdp_sym_2} is feasible and $\rho_k^G=\rho_k$.
As $\rho_k\to \rho^\ast$ from Proposition~\ref{CONV_MOMENTUM}, it follows that $\rho_k^G \to \rho^\ast$, when $k\to \infty$.
\end{proof}
\end{thm}

The same result as Proposition~\ref{CONV_MOMENTUM} can be proved for the symmetric relaxation, as a direct consequence of~\cite[Theorem 5]{TACCHI22}.

\begin{prop}\label{CONV_MOM}
Assume that Assumption~\ref{ass:nogap} and~\ref{ass:condition} hold and that Problem~\eqref{opt:relaxed_bis} has a unique solution $(\mu,\mu_T)$, i.e., that $\mathcal{M}^G$ is a singleton. Then any sequences $(z_{s \alpha \beta}^k)_{k\geq k_0}$ and $(y_{\alpha}^k)_{k\geq k_0}$ optimal for $(Q_k^G)$, we have, for every $s\in \N$, $\alpha\in \N^n$ and  $\beta\in\N^m$,
\[z_{s\alpha\beta}^k \to \int_{[0,T]\times X \times U} t^{s}x^{\alpha}u^{\beta} d\mu(t,x,u),\] \[y_{\alpha}^k \to \int_{K} x^{\alpha} d\mu_T(x),\] when $k\to +\infty$. 
\end{prop}



\section{A recovery method without symmetry reduction: how to catch extreme points in the set of optimal measures~\label{LINEAR_FUNK}}
In this section, we address the problem of recovering an approximation to an optimal trajectory from a given moment sequence obtained from solving~(\ref{sdp_sym_2}). We start with the case where symmetries are not exploited. Subsequently, we will tackle the recovery problem in the symmetry-reduced case.
 
 \subsection{Christoffel-Darboux kernel for curve reconstruction}
Given a moment sequence, a method to recover the support of the corresponding measure is to use the so-called Christoffel-Darboux kernel, following the steps of~\cite{MARX21} and~\cite[Section 6]{HL22}.
Let $b(t,y)$ be a polynomial basis of $\R[t,y]_k$, $T>0$ and let $Y\subseteq \R^n$ be an arbitrary set, and $f$ be a mapping from $[0,T]$ to $Y$. 
Let $M_k=\int_0^T b(t, f(t)) b(t, f(t))^T dt$ be the moment matrix of order $d=2k$ of the measure
$d\mu(t,y)=\delta_{f(t)}(dy)dt$ on $[0,T]\times Y$.
\begin{definition}\label{CRICRI}
 Define the Christoffel-Darboux (CD) polynomial
\[q_k(t,y)=b(t,y)(M_k+\beta_k \text{Id})^{-1}  \ b(t,y)^T,\] with a regularization parameter $\beta_k=2^{3-\sqrt k}$, and $f_k(t)=\text{min}\left( \text{argmin}_{y\in Y} q_k(t,y) \right)$.
\end{definition}
We have the following important result, which corresponds to~\cite[Theorem 1]{MARX21}.
\begin{thm}\label{CD_CONV}
If the set $S\subseteq [0,T]$ of continuity points of $f$ is such that $[0,T]\setminus S$ has Lebesgue measure zero, then $f_k(t)\to f(t)$ when $k\to +\infty$ for a.e. $t\in [0,T]$, and $\norm{f_k-f}_{L^1([0,T])}\to 0$.
\end{thm}

In the case where there is a unique solution to Problem~\eqref{optimalcontrol}, this method allows us to rebuild (at least approximately) the optimal trajectory. 
Indeed, in the latter case, Proposition~\ref{CONV_MOMENTUM} allows to recover trajectories thanks to the knowledge of the approximate moment matrix of the corresponding occupation measure, in the sense of Definition~\ref{OCCUP_MEAS}.
However, the reconstruction is much trickier when uniqueness fails. In this case, by linearity of Problem~\eqref{opt:relaxed}, the moments which are obtained via moment relaxations may correspond to approximations of the moments of any superposition of measures belonging to the set $\mathcal{M}$.

\begin{remark}
Non-uniqueness is a natural consequence of the existence of symmetry for Problem~\eqref{optimalcontrol}, when there exists $g\in G$ such that $\tau(g)\neq \text{Id}_{\R^m}$, where $\tau$ is as in Definition~\ref{DEF_SYM}. In this case, if $(x(t),u(t))$ steers a $G$-invariant element $x_0\in X$ towards a $G$-invariant set $K$, then $(g(x(t)),\tau(g)(u(t)))$ satisfies the same property, for every $g\in G$.
\end{remark}

\subsection{Catching extreme points}\label{CATCHMEIFYOUCAN}
In this section, we propose a method that allows approximating the moments of an occupation measure solution of Problem~\eqref{opt:relaxed}, without exploiting symmetries.
Our method, based on the resolution of two SDPs, ensures convergence towards the moments of the extreme points of the optimal measure solutions $\mathcal{M}$, which correspond to occupation measures by Corollary~\ref{EXTR_POINT}.

The two steps are:
\begin{enumerate}
    \item Compute the truncated optimal cost $\rho_k$ via the resolution of the SDP~\eqref{sdp_dense};
    \item Solve the SDP~\eqref{sdp_selection1} defined below in order to obtain an approximation of the moments of a pair of occupation measures in $\mathcal{M}$.
\end{enumerate}


\begin{assumption}\label{UNIK}
Consider two polynomials $P\in \R[t,x,u]$ and $\tilde{P} \in \R[x]$ such that
\begin{equation}
\inf_{(\mu,\mu_T)\in \mathcal{M}}  \int_{[0,T] \times X \times U} P(t,x,u) d\mu(t,x,u)+\int_{K} \tilde{P}(x) d\mu_T(x)
\label{PB_LINEAR_FONC}
\end{equation}
 has a unique solution $(\mu^\star,\mu_T^\star)\in \mathscr{M}_+([0,T]\times X \times U)\times \mathscr{M}_+(K)$. 
\end{assumption} 
Notice that  the optimization in (\ref{PB_LINEAR_FONC}) takes place over $\mathcal{M}$, i.e., the set of optimal solutions to~(\ref{opt:relaxed}).

\begin{remark}\label{ERIKGEN}
The fact that the measures in $\mathscr{M}_+([0,T] \times X\times U)$ and $\mathscr{M}_+(K)$ are moment determinate guarantees the existence of $P$ and $\tilde{P}$. 
Moreover,
the genericity of the uniqueness property of Assumption~\eqref{UNIK} w.r.t. polynomials $P$ and $\tilde{P}$ of large enough degree can be obtained by a direct application of~\cite[Theorem 5]{BERNARD}. 
\end{remark}
Define $k_1:=\text{max}(k_0,\text{deg}(P),\text{deg}(\tilde{P}))$.
For $k\geq k_1$, let $\rho_k$ be the truncated optimal cost of~\eqref{sdp_dense}, and consider the following SDP, denoted~\eqref{sdp_selection1}:
\begin{equation}
\label{sdp_selection1}
\tag{$Z_k$}
 \begin{aligned}
\inf_{y,z} \quad & L_z(P)+L_y(\tilde{P}),\\
\textrm{s.t.} \quad & L_z(h)+L_y(H) \leq \rho_k,\\
&M_k(y),M_k(z)\succeq 0,\\
&M_{k-\lceil\text{deg}(v_j)/2\rceil}(v_j z(x))\succeq 0,\\
&M_{k-\lceil\text{deg}(\theta_j)/2\rceil}(\theta_j y)\succeq 0,\\
&M_{k-\lceil\text{deg}(w_j)/2\rceil}(w_j z(u))\succeq 0,\\
&M_{k-1}(t(1-t) z(t))\succeq 0,\\
&L_y(\phi)-L_z\left(\partial \phi/ \partial t+\langle \nabla_x \phi, f \rangle \right)=\phi(0,x_0), \ \forall \phi=(t^s x^\alpha) \in \R[t,x] \ \text{s.t.} \ s+|\alpha|\leq 2k+1-\text{deg}(f).
\end{aligned}
\end{equation}

We have the following:
\begin{proposition}\label{CATCH}
\begin{itemize}
\item Under Assumption~\ref{UNIK}, the optimal pair $(\mu^\star,\mu_T^\star)$ of Problem~\eqref{PB_LINEAR_FONC} is a pair of occupation measures, in the sense of Definition~\ref{OCCUP_MEAS}.
\item For any sequences $(z_{s\alpha \beta}^k)_{k\geq k_1}$ and $(y_{\alpha}^k)_{k\geq k_1}$ optimal for~\eqref{sdp_selection1}, we have, for every $s\in \N$, $\alpha\in \N^n$ and  $\beta\in\N^m$, $z_{s\alpha\beta}^k \to \int_{[0,T]\times X \times U} t^{s}x^{\alpha}u^{\beta} d\mu^\star(t,x,u)$, $y_{\alpha}^k \to \int_{K} x^{\alpha} d\mu_T^\star(x)$ when $k\to \infty$. 
\end{itemize}
\end{proposition}

\begin{proof}
In order to prove the first claim, as $\mathcal{M}$ is weak-$\star$ compact by Proposition~\ref{EXTREME} and the map \[\mathcal{M}\ni(\mu,\mu_T) \mapsto \left(\int_{[0,T] \times X \times U} P(t,x,u) d\mu(t,x,u),\int_{K} \tilde{P}(x) d\mu_T(x)\right)\] is continuous, it is a direct consequence of~\cite[Corollary 4.2]{barvinok2002course} that the solution of Problem~\eqref{PB_LINEAR_FONC} is reached at an extreme point of $\mathcal{M}$. The uniqueness assumption of the solution of~\eqref{PB_LINEAR_FONC} implies that it is reached at an extreme point of $\mathcal{M}$, and only at this point. Using Proposition~\ref{EXTREME}, we obtain that the solution of Problem~\eqref{PB_LINEAR_FONC} is reached at a pair of occupation measures, and the first claim of the proof is proved.

The proof of the second claim is adapted from the proof of~\cite[Theorem 8]{TACCHI22}, with a slight modification provided that the value of $\rho_k$ depends on $k\geq k_0$. It relies on the simple fact that the condition $(\mu,\mu_T)\in \mathcal{M}$ is equivalent to the fact that $(\mu,\mu_T)$ is solution of Equation~\eqref{Liouville} and  $\langle h,\mu \rangle +\langle H,\mu_T \rangle \leq \rho^\star$, where $\rho^\star$ is the optimal cost of Problem~\eqref{opt:relaxed}.
Consider the sequence of SDP relaxations, denoted by~\eqref{sdp_selection}, defined for $k\geq k_1$ as
\begin{equation}
\label{sdp_selection}
\tag{$\bar{Z}_k$}
 \begin{aligned}
\inf_{y,z} \quad & L_z(P)+L_y(\tilde{P}),\\
\textrm{s.t.} \quad & L_z(h)+L_y(H) \leq \rho^\star,\\
&M_k(y),M_k(z)\succeq 0,\\
&M_{k-\lceil\text{deg}(v_j)/2\rceil}(v_j z(x))\succeq 0,\\
&M_{k-\lceil\text{deg}(\theta_j)/2\rceil}(\theta_j y)\succeq 0,\\
&M_{k-\lceil\text{deg}(w_j)/2\rceil}(w_j z(u))\succeq 0,\\
&M_{k-1}(t(1-t) z(t))\succeq 0,\\
&L_y(\phi)-L_z\left(\partial \phi/ \partial t+\langle \nabla_x \phi, f \rangle \right)=\phi(0,x_0), \ \forall \phi=(t^s x^\alpha) \in \R[t,x] \ \text{s.t.} \ s+|\alpha|\leq 2k+1-\text{deg}(f).
\end{aligned}
\end{equation}
 

By \cite[Proposition 7]{TACCHI22}, the SDP~\eqref{sdp_selection1} admits a pair of minimizing sequences $(z_k,y_k)_{k\geq k_1}$. For $\gamma=(s,\alpha,\beta)\in \N^{1+n+m}$, define the sequence $(\hat{z}^k)_{k\geq k_1}$ by $\hat{z}_{\gamma}^k=z_{\gamma}^k$ for $|\gamma|\leq 2k$, and $\hat{z}_{\gamma}^k=0$ if $|\gamma|>2k$. Define similarly $(\hat{y}^k)_{k\geq k_1}$ by $\hat{y}_{\alpha}^k=y_{\alpha}^k$ for $|\alpha|\leq 2k$, and $\hat{y}_{\alpha}^k=0$ if $|\alpha|>2k$.
By \cite[Lemma 5]{TACCHI22}, we obtain that $(\hat{z}^k)_{k\geq k_1}$ is uniformly bounded in $l^{\infty}(\N^{1+n+m})$ and $(\hat{y}^k)_{k\geq k_1}$ is uniformly bounded in $l^{\infty}(\N^{n})$, so that they admit weak-$\star$ converging subsequences. Denote the corresponding limits by $z_{\gamma,\infty}$ for every $\gamma \in \N^{1+n+m}$ and $y_{\alpha,\infty}$ for every $\alpha \in \N^{n}$.
As $\rho_{k}\to \rho^\star$ when $k\to +\infty$, we can show 
that the pair $\left((z_{\gamma,\infty})_{\gamma \in \N^{1+n+m}},(y_{\alpha,\infty})_{\alpha\in \N^n}\right)$  is feasible for the SDP~\eqref{sdp_selection}, for every $k\geq k_1$. 
Thanks to Theorem~\ref{PUTINAR_THM}, we obtain that $(z_{\gamma,\infty})_{\gamma\in \N^{1+n+m}}$ and $(y_{\alpha,\infty})_{\alpha\in \N^n}$ are the moment sequences of a pair of measures $(\mu^{\infty}, \mu_{T}^\infty) \in \mathscr{M}_+([0,T]\times X \times U)\times \mathscr{M}_+(K)$ that is feasible for Problem~\eqref{PB_LINEAR_FONC}. Moreover, by construction of the moment relaxations~\eqref{sdp_selection1} and~\eqref{sdp_selection}, we can check that $(\mu^{\infty},\mu_{T}^\infty)$ is optimal for Problem~\eqref{PB_LINEAR_FONC}, so that we can deduce by uniqueness that $(\mu^{\infty},\mu_{T}^\infty)=(\mu^\star,\mu_T^\star)$.
It follows that the sequences $(\hat{z}^k)_{k\geq k_1}$ and $(\hat{y}^k)_{k\geq k_1}$ are bounded sequences having unique weak-$\star$ accumulation points, and the result is proved.
\end{proof}

\begin{remark}
The SDP~\eqref{sdp_selection1} corresponds to a minimization over the outer approximation of the optimal set of Problem~\eqref{opt:relaxed}, computed previously with SDP~\eqref{sdp_dense}. In accordance with Remark~\ref{ERIKGEN}, we will consider $k\geq k_1$ large enough and random polynomials $P$ and $\tilde{P}$ of degree smaller than $k$ for its practical implementations.
\end{remark}

\section{Recovery of optimal trajectories with symmetry reduction~\label{RECOV} when the optimal invariant measure is unique~\label{uniqueness_recovery}}



Now we present a method to recover trajectories in the case where Problem~\eqref{optimalcontrol} is $G$-invariant, assuming uniqueness of the solution of Problem~\eqref{opt:relaxed_bis}, i.e., that the set $\mathcal{M}^G$ is a singleton. In this case, one can apply Proposition~\ref{CONV_MOM} in order to approximate the moments of the unique measure in $\mathcal{M}^G$. It is then an important task to propose a recovery method of the occupation measures thanks to the latter moments. 
Our method uses invariant polynomials and brings about a significant reduction in the computational cost compared to the method of Section~\ref{LINEAR_FUNK}.
The results of this section will be combined with those of Section~\ref{LINEAR_FUNK} in order to tackle the general case in Section~\ref{SELECTO}.

\subsection{Algorithm}
\textbf{ALGORITHM $A_1$}
\begin{enumerate}
\item Compute a lower approximation $\rho_k^G$ of the optimal cost via~\eqref{sdp_sym_2}, and obtain an approximation of the moments of the unique $G$-invariant pair of measures $(\mu^\star,\mu_T^\star)\in \mathcal{M}^G$;
\item Recover the trajectory and control associated with a pair of occupation measures $(\mu,\mu_T)\in \mathcal{M}$ which is an extreme point of the set $\mathcal{M}$, by solving $(P_1)$ or $(P_2)$, detailed in what follows.
\end{enumerate}

The first step has already been detailed in Section~\ref{HYERES_ARCHI} and consists in computing $\rho_k^G$.
Now we describe the second step, i.e., how do invariant polynomials allow us to recover the moments of optimal trajectories.

\subsection{Invariant polynomials and constancy results~\label{INV_POLY_USE}}

We start by giving a basic definition of invariants adapted to our setting.
 \begin{definition}\label{INVARIANT_POLY}
 \begin{itemize}
     \item  We say that a polynomial $Q\in \R[x,u]$ is $G$-invariant if it belongs to the set
 \[\R[x,u]^G=\{Q\in \R[x,u] \mid \ \forall g\in G, \ Q(g(x),\tau(g)(u))=Q(x,u)\}.\] Define similarly $\R[x]^G=\{P\in \R[x] \mid \ \forall g\in G, \ P(g(x))=P(x)\}$ and $\R[u]^G=\{V\in \R[x] \mid \ \forall g\in G, \ P(\tau(g)(u))=V(u)\}$. 
\item We say that a polynomial $Q\in \R[x,u]$ is $G$-state/control invariant if $Q$ can be decomposed as $Q(x,u)=P(x)V(u)$, where $P(g(x))=P(x)$ and $V(\tau(g)(u))=V(u)$, for every $g\in G$.
 \end{itemize}
 \end{definition}
Next proposition states that we can express some important invariant quantities over the optimal set $\mathcal{M}$ for Problem~\eqref{opt:relaxed}.

\begin{proposition}\label{INV_MOM_GEN}
Assume that $\mathcal{M}^G$ is a singleton $(\mu^\star,\mu_T^\star)$, and let $Q\in \R[x,u]^G$ and $P\in \R[x]^G$. Then the mapping \[\mathscr{M}_+([0,T] \times X \times U)\times \mathscr{M}_+(K) \ni (\mu,\mu_T) \mapsto \left(\int_{[0,T]\times X \times U} t^s Q(x,u)d\mu(t,x,u), \int_{K} P(x) d\mu_T(x)\right)\] is constant on $\mathcal{M}$.
\begin{proof}
Assume by contradiction that there exist distinct pairs $(\mu_1,\mu_{T1})$ and $(\mu_2,\mu_{T2})$ in $\mathcal{M}^2$ such that \begin{equation}
\label{INEQ}
\int_{[0,T]\times X \times U} t^s Q(x,u) d\mu_1(t,x,u)\neq \int_{[0,T]\times X \times U} t^s Q(x,u) d\mu_2(t,x,u)
\end{equation}
for some $s\in \N$.

By $G$-invariance of $Q$, we have that for every measure $\mu \in \mathscr{M}_+([0,T]\times X \times U)$ and $g\in G$, \[\int_{[0,T]\times X \times U} t^s Q(x,u) dg_{\#}\mu(t,x,u)= \int_{[0,T]\times X \times U} t^s Q(x,u) d\mu(t,x,u),\] and hence \[\int_{[0,T]\times X \times U} t^s Q(x,u) dR(\mu_q)(t,x,u)= \int_{[0,T]\times X \times U} t^s Q(x,u) d\mu_q(t,x,u),\] for $q\in \{1,2\}$.
It follows from~\eqref{INEQ} that the two $G$-invariant measures $R(\mu_1)$ and $R(\mu_2)$ are distinct, which is a contradiction provided that $\mathcal{M}^G$ is a singleton.
The same argument holds for the measures $\mu_{T1}$ and $\mu_{T2}$ in $\mathscr{M}_+(K)$, by considering the integrals $\int_{K} P(x) d\mu_{T1}(x)$ and  $\int_{K} P(x) d\mu_{T2}(x)$ for $P\in \R[x]^G$, and we get the result.

\end{proof}
\end{proposition}

\begin{corollary}\label{INV_MOM}
Assume that $\mathcal{M}^G$ is a singleton $(\mu^\star,\mu_T^\star)$, and let $Q(x,u)=P(x)V(u)\in \R[x,u]^G$ be $G$-state/control invariant. Then the mapping \[\mathscr{M}_+([0,T] \times X \times U)\times \mathscr{M}_+(K) \ni (\mu,\mu_T) \mapsto \left(\int_{[0,T]\times X \times U} t^s P(x) V(u) d\mu(t,x,u), \int_{K} P(x) d\mu_T(x)\right)\] is constant on $\mathcal{M}$, for every $s\in \N$.
In particular, if $(x(t),u(t))$ is a solution of Problem~\eqref{optimalcontrol}, then the moments of the measure $d\nu(t,x,u)=\delta_{P(x(t))}(dx)\delta_{V(u(t))}(du)dt$ defined on the set $[0,T]\times P(X)\times V(U)\subset [0,T]\times \R^2$ are given by \[\int_{[0,T]\times X \times U} t^s P(x)^\alpha V(u)^\beta d\mu^*(t,x,u),\] and those of the measure $d\nu_T(t,x,u)=\delta_{P(x(T))}(dx)$ defined on $P(K)\subset \R$ are given by $\int_{K} P(x)^\alpha d\mu_T^*(x)$, for every $\alpha,\beta,s\in\N$.
\end{corollary}





Consider a family of homogeneous generators $(Q_q)_{q\in \{1,\dots,p\}}$ of the invariant ring $\R[x,u]^G$, where $p\geq 1$, i.e., the family $(Q_q)_{q\in \{1,\dots,p\}}$ satisfies $\R[x,u]^G=\R[Q_1,\dots,Q_p]$.
The existence of such generators is guaranteed by Hilbert's finiteness theorem, and Noether's bound theorem (see, for instance~\cite{BOOK_INVARIANT}) ensures that one can bound the degree of the polynomials $(Q_q)_{q\in \{1,\dots,p\}}$ by the order $|G|$ of $G$.

 \paragraph{Main tasks:}
 The problem of recovery of optimal trajectories can then be seen from two different viewpoints, both requiring the use of the invariant polynomials $(Q_q)_q$:
 \begin{itemize}
     \item $(P_1)$: Find an admissible trajectory $(x(t),u(t))$ for Equation~\eqref{optimalcontrol}, assuming the knowledge (at least approximate) of the curves $(Q_q(x(t),u(t)))_{q\in \{1,\dots,p\}}$.
 \item $(P_2)$: Find the moments of a pair $(\mu,\mu_T)\in \mathcal{M}$ defined as occupation measures under the form \[\Bigl(d\mu(t,x,u),d\mu_T(x)\Bigr)=\Bigl(\delta_{x(t)}(dx)\delta_{u(t)}(du)dt,\delta_{x(T)}(dx)\Bigr),\] having an approximate knowledge of those of $(\mu^\star,\mu_T^\star)$, by a similar method of the one described in Section~\ref{CATCHMEIFYOUCAN}.
 \end{itemize}

\subsection{Problem~$(P_1)$}

 For an optimal trajectory $(x(t),u(t))$ of Problem~\eqref{optimalcontrol}, consider the curves \begin{equation}\label{INVERT_LIFT_general}
 z_q(t)=Q_q(x(t),u(t)),
 \end{equation}
  for $q\in \{1,\dots,p\}$, and every $t\in [0,T]$.
By Proposition~\ref{INV_MOM_GEN}, the moment matrices of the measures $d\nu_q(t,z)=\delta_{z_q(t)}(dz)dt$ are independent from the choice of the optimal trajectory $(x(t),u(t))$ and can be computed thanks to the moments of $(\mu^\star,\mu_T^\star)$, obtained solving Problem~\eqref{opt:relaxed_bis} via the resolution of the SDP~\eqref{sdp_sym_2}.
As a direct consequence, the curves $(z_q(t))_q$ can be recovered approximately via the use of Christoffel-Darboux kernels, as defined in Definition~\ref{CRICRI}.

 
\subsubsection{Solve the polynomial system~\eqref{INVERT_LIFT_general} in $(x(t),u(t))$ with state/control separation}
  Due to the fact that controlled vector field $f$ is polynomial and defined on a compact set, we have the existence of a solution $(x(\cdot),u(\cdot))$ of Equation~\eqref{INVERT_LIFT_general} such that the state $x(\cdot)$ is a Lipschitz function. It is then natural separate the state and the control and get rid at first of the control $u(t)$ which is possibly discontinuous (for instance in the case of bang-bang optimal controls). 
Consider a family of homogeneous generators $(P_{q_1},V_{q_2})_{(q_1,q_2)\in \{1,\dots,p_1\}\times \{1,\dots,p_2\}}$ of the invariant rings $\R[x]^G$ and $\R[u]^G$, where $p_1,p_2\geq 1$, so that Equation~\eqref{INVERT_LIFT_general} can be simplified into \begin{equation}\label{INVERT_LIFT}
 (y_{q_1}(t),v_{q_2}(t))=\left(P_{q_1}(x(t)),V_{q_2}(u(t))\right),
 \end{equation}
  for $(q_1,q_2)\in \{1,\dots,p_1\}\times \{1,\dots,p_2\}$, and every $t\in [0,T]$, which is divided into separate equations on the variables $x$ and $u$.
A method to test the dynamical feasibility of a solution of Equation~\eqref{INVERT_LIFT} and, in the positive case, recover the associated optimal control is proposed in Appendix~\ref{CONTROL_RECO}.
  
  \begin{remark}
 By Corollary~\ref{INV_MOM}, the moments of the measure
 \[d\nu_{q_1,q_2}(t,x,u)=\delta_{y_{q_1}(t)}(dx)\delta_{v_{q_2}(t)}(du)dt\] defined on $[0,T]\times \R^2$ are given by \[\int_{[0,T]\times X \times U} t^s P_{q_1}(x)^\alpha V_{q_2}(u)^\beta d\mu^*(t,x,u),\] for every $(q_1,q_2)\in \{1,\dots,p_1\}\times \{1,\dots,p_2\}$, $s,\alpha,\beta \in \N$.
 \end{remark}

\begin{remark}\label{INVSYSSOL}
In the case of sign-symmetries given in Example~\ref{EXX},
the set $\R[x]^G$ is generated by $P(x)=x^2$ and $\R[u]^G$ is generated by $V(u)=u^2$.
Hence solving~\eqref{INVERT_LIFT} in $x$-component is equivalent to taking the square root, up to the singularity at $x=0$. Solving~\eqref{INVERT_LIFT} in the $u$-component may be harder due to possible discontinuities of the control $u(t)$.
\end{remark}

   \begin{remark}
  In this setting, Problem~$(P_1)$ corresponds to a lift over invariant problem (see for instance \cite{PR16,PR21} for a study of regularity issues). The resolution of  Equation~\eqref{INVERT_LIFT} can be made analytically in some easy cases (e.g. sign symmetries). Computer algebra methods can be used in more complex cases, which are, however, beyond the scope of this paper.
  \end{remark}

\subsection{Problem~$(P_2)$~\label{SUPP:TRICK}}

In the case where solving~$(P_1)$ is too hard, 
we propose a numerical method involving a supplementary step, which aims at getting the moments of an extreme point $(\mu,\mu_T)$ of the set $\mathcal{M}$ using the approximate values of the moments of $(\mu^\star,\mu_{T}^\star)$, and the selection of extreme measures proved in Section~\ref{CATCHMEIFYOUCAN} with SDP~\eqref{sdp_selection1}. 
It consists in the minimization of a functional built from polynomials $P\in \R[t,x,u]$ and $\tilde{P}\in \R[x]$ guaranteeing uniqueness of the following LP on measures:
\begin{equation} \label{opt:lift}
\begin{aligned}
 \inf\limits_{\mu,\mu_T} \quad & \int P\,d\mu+\int \tilde{P}d\mu_T\\
\textrm{s.t.} \quad &   (\mu,\mu_T) \text{\;\,satisfy (\ref{Liouville})}\\
  &(\mu,\mu_T) \in \mathcal{A}_{\mu^\star,\mu_T^\star},   \\
\end{aligned}
\end{equation}
where $\mathcal{A}_{\mu^\star,\mu_T^\star}$ is defined as in Definition~\ref{MEAS_DEFI}.

By noticing that $(\mu,\mu_T) \in \mathcal{A}_{\mu^\star,\mu_T^\star}$ if and only if $\int t^s Q(x,u)d\mu=\int t^s Q(x,u)d\mu^\star$ and $\int \tilde{Q}(x)d\mu_T=\int \tilde{Q}(x)d\mu_T^\star$ for every $s\in \N$ and $(P,\tilde{Q})\in \R[x,u]^G\times \R[x]^G$, we can propose the following SDP relaxation sequence, denoted by~\eqref{sdp_lift} and defined, for $k\geq k_1$, as
\begin{equation}
\label{sdp_lift}
\tag{$R_k$}
 \begin{aligned}
\inf_{y,z} \quad & L_z(P)+L_y(\tilde{P})\\
\textrm{s.t.} \quad & L_z(h)+L_y(H)\leq \rho_{k}^G\\
 &M_k(y),M_k(z)\succeq 0,\\
&M_{k-\lceil\text{deg}(v_j)/2\rceil}(v_j z(x))\succeq 0,\\
&M_{k-\lceil\text{deg}(\theta_j)/2\rceil}(\theta_j y)\succeq 0,\\
&M_{k-\lceil\text{deg}(w_j)/2\rceil}(w_j z(u))\succeq 0,\\
&M_{k-1}(t(1-t) z(t))\succeq 0,\\
&L_y(\phi)-L_z\left(\partial \phi/ \partial t+\langle \nabla_x \phi, f \rangle \right)=\phi(0,x_0), \ \forall \phi=(t^s x^\alpha) \in \R[t,x] \ \text{s.t.} \ s+|\alpha|\leq 2k+1-\text{deg}(f),\\
& L_z(t^s P_q(x)^\alpha V_q(u)^\beta)=L_{z^\star_k}(t^s P_q(x)^\alpha V_q(u)^\beta),\\
& L_y(P_q(x)^\alpha )=L_{y^\star_k}(P_q(x)^\alpha),\\
&\text{for every} \ s\in \N, \alpha\in \N, \beta\in\N, q\in \{1,\dots,p\},\\
\end{aligned}
\end{equation}
where $z^\star_k, y^\star_k$ are the optimal solutions of the SDP~\eqref{sdp_sym_2}.

By similar arguments to those used in the proof of Proposition~\ref{CATCH} concerning the convergence of SDP~\eqref{sdp_selection1}, we obtain the following convergence result.
\begin{proposition}
For any sequences $(z_{s\alpha \beta}^k)_{k\geq k_0}$ and $(y_{\alpha}^k)_{k\geq k_0}$ optimal for~\eqref{sdp_lift}, we have, for every $s\in \N$, $\alpha\in \N^n$ and  $\beta\in\N^m$, \[z_{s\alpha\beta}^k \to \int_{[0,T]\times X \times U} t^{s}x^{\alpha}u^{\beta} d\mu(t,x,u),\] and \[y_{\alpha}^k \to \int_{K} x^{\alpha} d\mu_T(x),\] when $k\to \infty$, where $(\mu,\mu_T)$ is a pair of occupation measures solution of~\eqref{opt:relaxed}.
\end{proposition}

\begin{remark}
Note that even if this method breaks symmetry and does not involve the symmetry reduced moment matrices, the number of SDP variables is significantly reduced because of multiple moment substitutions. The efficiency of~\eqref{sdp_lift} will be illustrated in the numerical results of Section~\ref{APPLIQ}.
\end{remark}

\section{Recovery of trajectories with symmetry reduction in the general case~\label{SELECTO}}

Our next Algorithm~$(A_2)$ allows to tackle the case where the solution of Problem~(\ref{opt:relaxed_bis}) is not unique, i.e., the set $\mathcal{M}^G$ is not a singleton, using symmetry reduction. It requires the supplementary minimization Step~$2$, which relies on the results of Section~\ref{LINEAR_FUNK} and generalizes the hierarchy~\eqref{sdp_selection1} to the symmetric setting.

\textbf{ALGORITHM $A_2$}
\begin{enumerate}
\item Compute a lower approximation $\rho_k^G$ of the optimal cost via~\eqref{sdp_sym_2};
\item Select an extreme pair of $G$-invariant measures $(\mu^\star,\mu_{T}^\star)\in \mathcal{M}^G$, by minimizing a random symmetric linear functional;
\item Apply Step~$2$ of Algorithm~$(A_1)$ (i.e. solve $(P_1)$ or $(P_2)$) in order to recover the trajectory and control associated with a pair of occupation measures $(\mu,\mu_T)\in \mathcal{M}$ belonging to the set $\mathcal{A}_{\mu^\star,\mu_T^\star}$, where $\mathcal{A}_{\mu^\star,\mu_T^\star}$ is defined as in Definition~\ref{MEAS_DEFI}.
\end{enumerate}

The goal of what follows is to describe Step~$2$ of Algorithm~$(A_2)$.
Indeed, the first  step (respectively, third step) has already been described in Section~\ref{HYERES_ARCHI} (respectively, Section~\ref{uniqueness_recovery}).
Contrary to the case described in Section~\ref{RECOV} where $\mathcal{M}^G$ is a singleton, the mapping \[(\mu,\mu_T) \mapsto \left(\int_{[0,T]\times X \times U} t^s P(x)^\alpha V(u)^\beta \mu(t,x,u),\int_{K} P(x)^\alpha \mu_T(x)\right)\] is in general not constant for $(\mu,\mu_T) \in \mathcal{M}$. However it is constant on every $\mathcal{A}_{\mu,\mu_T}$, for $(\mu,\mu_T)\in \mathcal{M}^G$, where $\mathcal{A}_{\mu,\mu_T}$ is as in Definition~\ref{MEAS_DEFI}. The goal is then to recover the value of this mapping on $\mathcal{A}_{\mu^\star,\mu_T^\star}$, where $(\mu^\star,\mu_T^\star)$ is an extreme point of $\mathcal{M}^G$. 

\begin{prop}\label{INV_MOM_BIS}
Let $Q(x,u)=P(x)V(u)\in \R[x,u]^G$ be $G$-state/control invariant, in the sense of Definition~\ref{INVARIANT_POLY}.
 The moments of any measure of the form \[\left(d\nu(t,x,u),d\nu_T(x)\right)=\left(\delta_{P(x(t))}(dx)\delta_{V(u(t))}(du)dt,\delta_{P(x(T))}(dx)\right),\] where $(x(t),u(t))$ are optimal pairs of Problem~\eqref{optimalcontrol} are given by \[\left(\int_{[0,T]\times X \times U} t^s P(x)^\alpha V(u)^\beta d\mu(t,x,u),\int_{K} P(x)^\alpha d\mu_T(t,x,u)  \right)\] where $(\mu,\mu_T)$ is an extreme point of $\mathcal{M}^G$.
\begin{proof}
By Proposition~\ref{INVARIANCE}, the extreme points of $\mathcal{M}^G$ are \[\mathcal{E}_G=\{(R(\mu),R_T(\mu_T)) \mid \ (\mu,\mu_T)\in \mathcal{M} \ \text{is a pair of occupation measures}\}.\]
Consider an extreme point of $\mathcal{M}^G$ which can be written as the image $R(\mu)$ by $R$ of an occupation measure $\mu$. Then we have
\[\begin{aligned}
    \int_{[0,T]\times X \times U} t^s P(x)^\alpha V(u)^\beta dR(\mu)
    =&\frac{1}{|G|} \sum_{g\in G} \int_{[0,T]\times X \times U} t^s P(g(x))^\alpha V(\tau(g)(u))^\beta d\mu\\
    =&\int_{[0,T]\times X \times U} t^s P(x)^\alpha V(u)^\beta d\mu\\
    =&\int_{[0,T]\times X \times U} t^s P(x)^\alpha V(u)^\beta d\nu,  
\end{aligned}\]
where the second equality is obtained by $G$-invariance of $Q=PV$, and the third one by definition of $\nu$. The same argument being true for the measure $\nu_T$ by using $R_T$, the result follows.
\end{proof}

\end{prop}

\begin{assumption}\label{Gextr}
Consider $G$-invariant polynomials $P\in \R[t,x,u]$ and $\tilde{P} \in \R[x]$ such that
\begin{equation}
\inf_{(\mu,\mu_T)\in \mathcal{M}^G}  \int_{[0,T] \times X \times U} P(t,x,u) d\mu(t,x,u)+\int_{K} \tilde{P}(x) d\mu_T(x)
\label{PB_LINEAR_FONC_INV}
\end{equation}
 has a unique solution $(\mu^\star,\mu_T^\star)\in \mathcal{M}^G$, which is an extreme point of $\mathcal{M}^G$. 
\end{assumption} 

Consider the SDP relaxation sequence, denoted by~\eqref{sdp_selection_sym}, defined for $k\geq k_0$ as
 \begin{equation}
 \label{sdp_selection_sym}
 \tag{$Z_k^G$}
 \begin{aligned}
\inf_{y,z} \quad & L_z(P)+L_y(\tilde{P})\\
\textrm{s.t.} \quad & L_z(h)+L_y(H)\leq \rho_{k}^G\\
&M_k^G(y),M_k^G(z)\succeq 0,\\
&M_{k-\lceil\text{deg}(v_j)/2\rceil}^G(v_j z(x))\succeq 0,\\
&M_{k-\lceil\text{deg}(\theta_j)/2\rceil}^G(\theta_j y)\succeq 0,\\
&M_{k-\lceil\text{deg}(w_j)/2\rceil}^G(w_j z(u))\succeq 0,\\
&M_{k-1}^G(t(1-t) z(t))\succeq 0,\\
&L_y(\phi)-L_z\left(\partial \phi/ \partial t+\langle \nabla_x \phi, f \rangle \right)=\phi(0,x_0), \ \forall \phi=(t^s x^\alpha) \in \R[t,x] \ \text{s.t.} \ s+|\alpha|\leq 2k+1-\text{deg}(f).\\
\end{aligned}
\end{equation}
This SDP is a generalization of SDP~\eqref{sdp_selection1}.
By similar arguments to those used in the proof of Proposition~\ref{CATCH} together with Proposition~\ref{INV_MOM_BIS}, we obtain the following convergence result.
\begin{proposition}
Assume that Assumption~\ref{Gextr} holds. Then for any sequences $(z_{s\alpha \beta}^k)_{k\geq k_0}$ and $(y_{\alpha}^k)_{k\geq k_0}$ optimal for~\eqref{sdp_selection_sym}, we have, for every $s\in \N$, $\alpha\in \N^n$ and  $\beta\in\N^m$, \[z_{s\alpha\beta}^k \to \int_{[0,T]\times X \times U} t^{s}x^{\alpha}u^{\beta} d\mu^\star(t,x,u),\] and \[y_{\alpha}^k \to \int_{K} x^{\alpha} d\mu_T^\star(x),\] when $k\to \infty$.
\end{proposition}





\section{Numerical examples}\label{APPLIQ}

We illustrate our methods first with an elementary example of an integrator with symmetries.
Then we tackle the case of qubit inversion, which is the most simple example of quantum systems, and which has attracted a large interest for years~\cite{BM06,gatto}.
All computations shown were run on a personal computer running MACOS with an
11th Gen Intel(R) Core(TM) i7-11800H @ 2.20GHz Processor and 16GB of RAM. 
The software was coded in Matlab utilizing the Gloptipoly3 \cite{henrion2009gloptipoly} library for problem formulation or Yalmip~\cite{YALMIP}, and SeDuMi \cite{sturm1999using} as the SDP solver.\footnote{See the code archive \url{https://github.com/nicolasaugier1/SYMMETRIC_OCP.git} for more details.}
The code~\url{symmetric_OCP_SOS.m} achieves the cost computation with Yalmip modeling and uses the classical dual formulation on polynomials of the Moment hierarchies stated in the present paper (see e.g.~\cite{OCP08,SYMSDP,TACCHI22}), while~\url{symmetric_OCP_qubit.m} achieves the trajectory reconstruction via Algorithm $(A_2)$. 
The use of Yalmip is restricted to the computation of optimal costs, corresponding to the first step of the proposed algorithms solving SDP~\eqref{sdp_sym_2}.
As Gloptipoly modeling is more convenient for dealing with the measure formulation of our optimization problems, we will use it for trajectory recovery. However, as a counterpart we will restrict the numerical simulations concerning trajectory reconstructions to the symmetry reduction~\eqref{sdp_sym_1} described in Appendix~\ref{A22} which considers $G$-invariant moment sequences. It achieves in practice only moment substitutions instead of block diagonal reduction of moment and localizing matrices. 


\subsection{Toy model: integrator with symmetry~\label{INTEGRATOR}}

Consider the following integrator optimal control problem:
\begin{equation}
\tag{INT}\label{INTEG}
\begin{aligned}
     \min \quad & T\\
    \textrm{s.t.}\quad &\dot{x}(t)=u(t)\;\;\text{on}\;\; [0,T]\\
    &|u(t)|\leq 1,\;|x(t)|\leq 1\\
    &x(0)=0,\;x(T)^2=1.
\end{aligned}
\end{equation}
Here we have $X=[-1,1]=\{x\in \R \mid v_1(x)=1-x^2\geq 0\}$, $U=[-1,1]=\{u\in \R \mid w_1(u)=1-u^2\geq 0\}$, $K=\{\pm 1\}=\{x\in \R \mid \theta_1(x)=1-x^2=0\}$, which are compact semi-algebraic sets such that Assumption~\ref{ass:condition} is satisfied. Moreover, one can check easily that Assumption~\ref{ass:nogap} is satisfied.
The group $G=\{-1,1\}$ together with multiplication acts as a sign symmetry with $1(x)=x$, $-1(x)=-x$, $\tau=\text{Id}_{\R}$, in accordance with Definition~\ref{DEF_SYM}. The initial state and final set are $G$-invariant, in the sense of Definition~\ref{group_inv_def}.
By a direct analysis of the dynamics, we show that the optimal cost is $T=1$ and the two optimal trajectories are defined, for every $t\in [0,1]$, by $(x_1(t),u_1(t))=(t,1)$ and $(x_2(t),u_2(t))=(-t,-1)$. Their associated occupation measures are $d\nu_j(t,x,u)=\delta_{x_j(t)}(dx)\delta_{u_j(t)}(du)dt$, for $j\in \{1,2\}$.
By Proposition~\ref{EXTR_POINT}, the set $\mathcal{M}$ of optimal measures for Problem~\eqref{opt:relaxed} is equal to \[\mathcal{M}=\{\left(\lambda_1 d\nu_1(t,x,u)+\lambda_2 d\nu_2(t,x,u),\lambda_1 \delta_1(dx)+\lambda_2 \delta_{-1}(dx)\right) \mid \lambda_1+\lambda_2=1, \ \lambda_1\geq 0, \lambda_2\geq 0\},\] so that the unique $G$-invariant pair of measures in $\mathcal{M}$ is $(\mu^\star,\mu_T^\star)$, where $\mu^\star=\frac{1}{2}(\nu_1+\nu_2)$ and $\mu_T^\star=\frac{1}{2}(\delta_1(dx)+\delta_{-1}(dx))$.

\subsubsection{Cost computation}
On Table~\ref{BETA_1}, we compare the dense~\eqref{sdp_dense} and symmetric~\eqref{sdp_sym_2} relaxations in the case $d=2k\in \{14,16,18,20\}$, where the SDPs are modeled with Yalmip. We notice a better reduction of the computational time due to the block structure of the semi-definiteness constraints involved in the SDPs for~\eqref{sdp_sym_2}, and to the reduced SDP variables.
The fourth column corresponds to the number of moments involved in the SDP. 
The reduction by a factor 2 of the number of pseudo-moment variables in the symmetry-adapted SDP relaxation is consistent with our theoretical estimates from Section~\ref{SYM_MOM}.




\begin{table}[!ht]
 \caption{Comparison of dense~\eqref{sdp_dense} and symmetric~\eqref{sdp_sym_2} relaxations for integrator system, modeled with Yalmip with $d=2k\in \{14,16,18,20\}$.}
\begin{center}
   \begin{tabular}{ | l | c | c | r | }
     \hline
     $d=14$ & cost & time & SDP variables \\ \hline
    Without symmetry & 0.9736 & 13s & 709   \\ \hline
    With symmetry  & 0.9740 & 4.8s & 359  \\
     \hline
     $d=16$ &  &  &  \\ \hline
    Without symmetry & 0.9740 & 28.7s & 1002  \\ \hline
    With symmetry  & 0.9748 & 9.5s & 506  \\
     \hline
    $d=18$  &  &  &  \\ \hline
    Without symmetry & 0.9752 & 81.3s & 1367 \\ \hline
    With symmetry   &  0.9760 & 20.8s & 689  \\
     \hline
    $d=20$  &  &   & \\ \hline
    Without symmetry & 0.9736 & 183s & 1812  \\ \hline
    With symmetry   & 0.9755 & 44.54s & 912 \\
    \hline
   \end{tabular}
 \end{center}
  \label{BETA_1}
\end{table}

\subsubsection{Reconstruction via solving $(P_1)$ in Algorithm~$(A_1)$}
On this toy problem, the solution of~\eqref{opt:relaxed_bis} is unique and equal to $(\mu^\star,\mu_T^\star)$, and we can apply Algorithm~$(A_1)$ described in Section~\ref{uniqueness_recovery}.
We consider the invariant polynomial $Q(x)=x^2$, and thanks to Proposition~\ref{INV_MOM}, we have $\int_{[0,T]} Q(x(t))dt=\int Q(x)d\mu^\star(t,x,u)$. In this setting, the polynomial system~\eqref{INVERT_LIFT} restricted to the state trajectory writes $y(t)=Q(x(t))$, for every $t\in [0,T]$.
This fact allows plotting the image $Q(x(t))$ of one of the two optimal trajectories $x(t)$ as a function of $t$ obtained via Christoffel-Darboux kernels on Figure~\ref{squared}. 
In this case, approximations of the two optimal trajectories can be obtained by taking the square root, modulo a sign choice at the points where $x(t)$ vanishes (here it is only the case at $t=0$, as it can be guessed numerically from the curve of $Q(x(t))$). 

\begin{figure}[!ht]
    \centering
    \subfigure[Rebuilt trajectory of $(x(t))^2$]{\label{split} \includegraphics[scale=0.55]{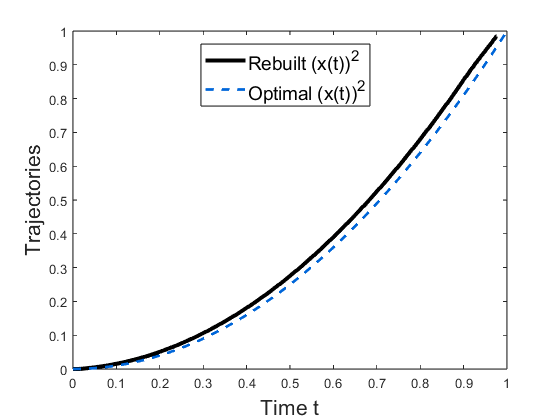}}
    \subfigure[Rebuilt trajectory of $x(t)$]{\label{adm:fig} \includegraphics[scale=0.55]{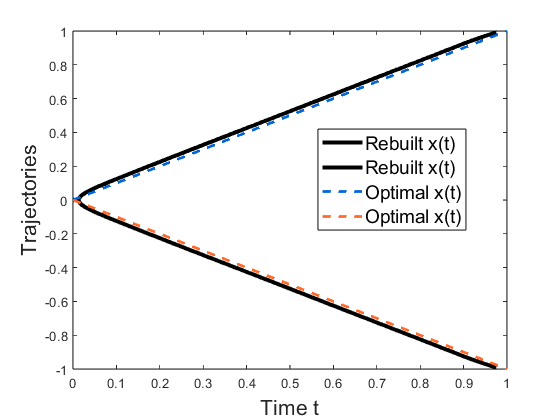}}
    \caption{Rebuilt and theoretical trajectories of Problem~\eqref{INTEG} as a function of $t$, computed with $d=16$ via Christoffel-Darboux Kernels.}
    \label{squared}
\end{figure}

As explained in Section~\ref{INV_POLY_USE}, the control reconstruction is harder because the control can be discontinuous. 
In particular, considering the invariant polynomial $V(u)=u^2$, we have $\int_{[0,T]} V(u(t))dt=\int V(u)d\mu^\star(t,x,u)$. 
This allows us to plot the image $V(u(t))$ as a function of $t$ obtained via Christoffel-Darboux kernels on Figure~\ref{squared:control}. We can guess numerically that the optimal control should be bang-bang. However,  Figure~\ref{squared:control} does not allow recovering $u(t)$ due to its possible switches between the values $\pm 1$.
\begin{figure}[H]
\begin{center}
    \includegraphics[scale=0.55]{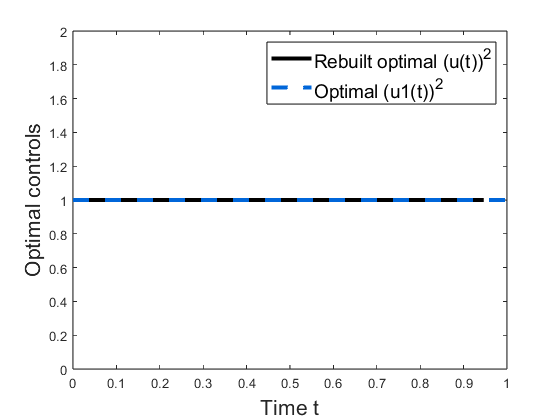}
 \end{center}
 \caption{Rebuilt and theoretical image via $V(u)=u^2$ of optimal control trajectories of Problem~\eqref{INTEG} as a function of $t$, computed with $d=16$ via Christoffel-Darboux Kernels.}
 \label{squared:control}
\end{figure}

\subsubsection{Optimal control reconstruction via solving $(P_2)$ in Algorithm~$(A_1)$}  

 In order to recover an optimal control, we propose to solve the LP on measures~\eqref{opt:lift} via the corresponding moment relaxation~\eqref{sdp_lift}, as illustrated on Figure~\ref{syncadmin}. This method involves only moment substitutions without block diagonalization of moment matrices as in SDP~\eqref{sdp_sym_2}. The computation time is equal to $12,9506s$ with $498$ pseudo-moment variables for $d=16$.


\begin{figure}[H]
    \centering
\includegraphics[scale=0.55]{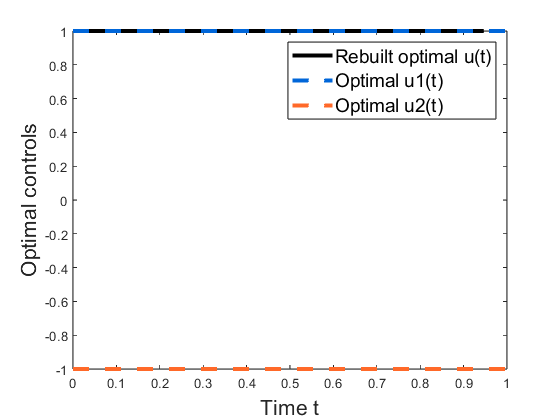}
    \caption{Reconstruction of the control $u(t)$ thanks to the hierarchy~\eqref{sdp_lift}, with $d=2k=16$.}
    \label{syncadmin}
\end{figure}



\subsection{Minimal time inversion of a single qubit}
Consider the following control system, corresponding to the equation of a Qubit on the Bloch Sphere $\mathbb{S}^2$: 
\begin{equation}\label{sistema1}
\dot{x}(t)=(A_0+u(t)A_1)x(t)
\end{equation}
where $x(t)$ belongs to the real two-dimensional sphere $\mathbb{S}^2 \subset \R^3$.
The matrices \[A_0=\begin{pmatrix}
0&-\kappa \cos(\alpha)&0\\
\kappa \cos(\alpha)&0&0\\
0&0&0
\end{pmatrix}
,\]
 
\[A_1=\begin{pmatrix}
0&0&0\\
0&0&-\kappa \sin(\alpha)\\
0&\kappa\sin(\alpha)&0
\end{pmatrix}
,\] are $3\times 3$ real antisymmetric matrices, with given parameters $\kappa\geq 0$ and $\alpha\in (0,\pi/2)$.
The OCP is the following: 
\begin{equation}
\label{opt}
\tag{QUBIT}
\begin{aligned}
	\min \quad & T\geq 0,  \\
\textrm{s.t.} \quad	& \dot{x}(t)=(A_0+u(t)A_1)x(t) \;\;\text{on}\;\; [0,T] \\ 
 & x(t)\in \mathbb{S}^2,\;  u(\cdot) \in [-1,1]\\
 &x(0)=(0,0,1),\; x(T)=(0,0,-1).
\end{aligned}
\end{equation}

With the notations of Section~\ref{FRAME}, we have $X=\mathbb{S}^2=\{x\in \R^3 \mid x_1^2+x_2^2+x_3^2-1\geq 0, \  -(x_1^2+x_2^2+x_3^2)-1\geq 0\}$, $U=[-1,1]=\{u\in \R \mid w_1(u)=1-u^2\geq 0\}$, and $K=(0,0,-1)$, which are compact semi-algebraic sets, satisfying Assumption~\ref{ass:condition}. Moreover, one can check easily that Assumption~\ref{ass:nogap} is satisfied.
The group $G=\{D,\text{Id}_{\R^3}\}$, with $D=\text{diag}(-1,-1,1)$ together with matrix multiplication, acts as a sign symmetry with $\tau(\text{Id}_{\R^3})=1$, $\tau(D)=-1$, in accordance with Definition~\ref{DEF_SYM}.
Notice that $x_0=(0,0,1)$ and $x_1=(0,0,-1)$ are $G$-invariant, in the sense of Definition~\ref{group_inv_def}.
Using the theoretical results of~\cite{BM06}, when $\alpha\geq \pi/4$, we know that there are four optimal control strategies. The optimal controls consist of two bang arcs, with switching times occuring when the trajectory $x(t)$ crosses the equator of $\mathbb{S}^2$, and it is proved analytically that the minimal time is equal to $T_f=2\pi/\kappa$.
More precisely, the four optimal control are $(u_1(t),-u_1(t),u_2(t),-u_2(t))$, where $u_1(t),u_2(t)$ are defined as:

\begin{align*}
u_1(t)\equiv \left\{\begin{array}{l}
1 \ \text{for} \ t\in [0,t_s]\\
-1 \ \text{for} \ t\in [t_s,T_f], \\
\end{array}\right.
\end{align*}
and 
\begin{align*}
u_2(t)\equiv \left\{\begin{array}{l}
1 \ \text{for} \ t\in [0,T_f-t_s]\\
-1 \ \text{for} \ t\in [T_f-t_s,T_f], \\
\end{array}\right. 
\end{align*}
where $t_s=\pi-\arccos{(\cot^2 (\alpha))}$. 
Denote the trajectories associated to those controls by $(x^j(t))_{j\in \{1,\dots,4\}}$, belonging to $\mathbb{S}^2$.
The controls $u_1(t),u_2(t)$ are plotted on Figure~\ref{MERE}. For readability, we did not plot the symmetric controls $-u_1(t),-u_2(t)$.
On the figures~\ref{ROUGE}, \ref{COULE}, and \ref{SOEUR}, we have plotted the $x_1,x_2,x_3$ components of the trajectories corresponding to the optimal controls $u_1(t),u_2(t)$. Note that taking the sign symmetric controls $-u_1(t),-u_2(t)$ achieves a change of sign for $x_1(t)$ and $x_2(t)$, while keeping $x_3(t)$ invariant. 
On Figure~\ref{spherical}, we have plotted the four possible optimal trajectories on the sphere $\mathbb{S}^2$.
\begin{figure}[!ht]
    \centering
    \subfigure[$x_1$-component with $u_1(t)$ or $u_2(t)$]{\label{ROUGE} \includegraphics[scale=0.54]{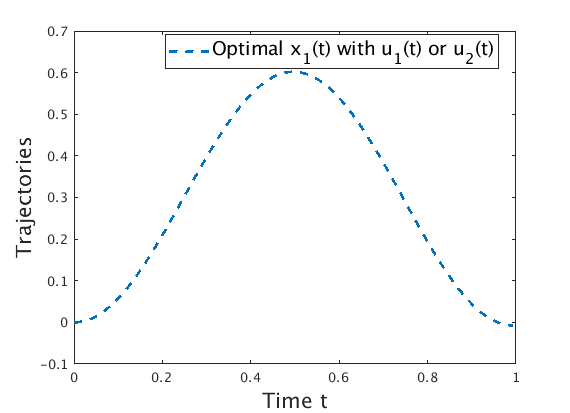}}
    \subfigure[$x_2$-component with $u_1(t)$ or $u_2(t)$]{\label{COULE} \includegraphics[scale=0.54]{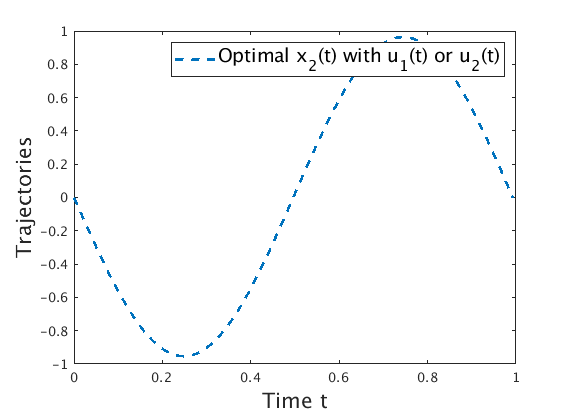}}
     \subfigure[$x_3$-component with $u_1(t)$ or $u_2(t)$]{\label{SOEUR} \includegraphics[scale=0.54]{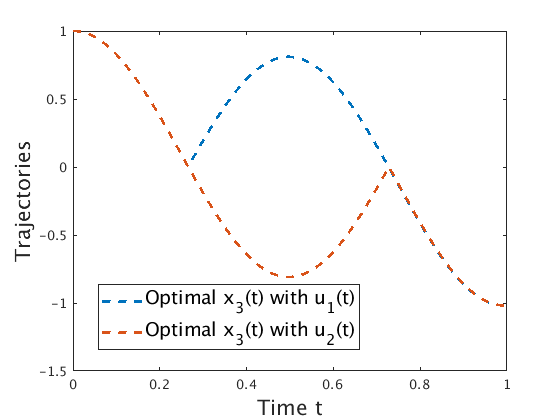}}
       \subfigure[Optimal controls $u_1(t),u_2(t)$]{\label{MERE} \includegraphics[scale=0.54]{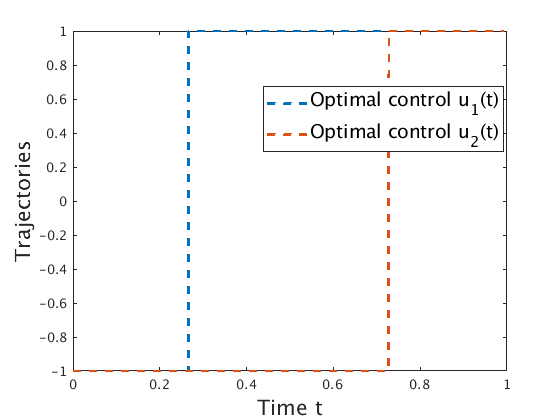}}
    \caption{Components of optimal trajectories and optimal controls modulo sign symmetry.}
    \label{syncadm_opti}
\end{figure}

\begin{figure}[!ht]
    \centering
    \subfigure[Optimal trajectories around the North pole. The pole is plotted as a red dot. ]{\label{ROUGES} \includegraphics[scale=0.55]{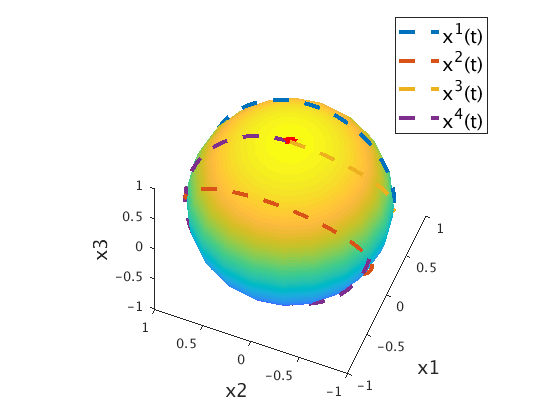}}
    \subfigure[Optimal trajectories around the South pole. The pole is plotted as a red dot.]{\label{COULES} \includegraphics[scale=0.55]{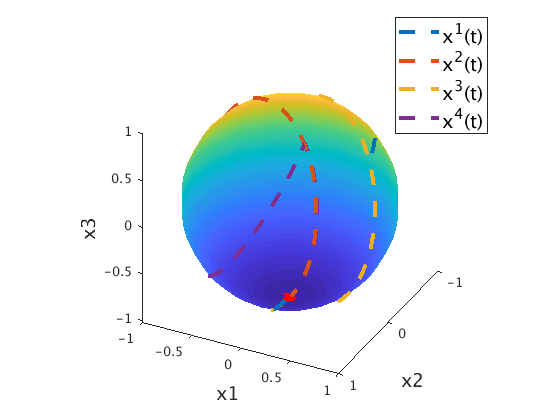}}
    \caption{The four optimal trajectories $(x^1(t),x^2(t),x^3(t),x^4(t))$ plotted on the sphere in different colors}
    \label{spherical}
\end{figure}

\subsubsection{Cost computations}
We assume that $\kappa=\frac{\sqrt{10}}{2}$, $\alpha=\arctan(3)\approx 1,25\geq \pi/4$. In this case, the theoretical minimal time is $\rho=2\pi/\kappa\approx 0.9935$.
Table~\ref{SPIN_EFFICIENCY_2} displays our numerical results, providing the lower bounds for the cost, with a relaxation order $d=2k\in \{6,8,10,16\}$, obtained respectively with the relaxations~\eqref{sdp_dense} and~\eqref{sdp_sym_2}, modeled with Yalmip.  As in the integrator we notice a better reduction of the computational time compared to previous case due to the block structure of the semi-definitness constraints involved in the SDPs for~\eqref{sdp_sym_2}.
Note that we have not been able to perform the experiments for $d=16$ without exploiting symmetry because our computational device ran out of memory as indicated by the abbreviation ``OoM'' in the table.

\begin{table}[!ht]
 \caption{Comparison of dense~\eqref{sdp_dense} and symmetric~\eqref{sdp_sym_2} relaxations for qubit system, with Yalmip modeling, with $d=2k\in \{6,8,10\}$.}
\begin{center}
   \begin{tabular}{ | l | c | c | r |}
     \hline
     $d=6$ & cost & time & SDP variables\\ \hline
    Without symmetry & 0.7708 & 1.7s & 672\\ \hline
    With symmetry  & 0.7708 & 1.7s & 346\\
     \hline
     $d=8$ &  &  &  \\ \hline
    Without symmetry & 0.8758 & 16.2s & 1782\\ \hline
    With symmetry  & 0.8758 & 5.5s & 906 \\
     \hline
    $d=10$  &  &  & \\ \hline
    Without symmetry & 0.9244 & 226s & 4004 \\ \hline
    With symmetry   & 0.9244 & 55.5s & 2023\\
     \hline
     $d=16$  &  &  & \\ \hline
    Without symmetry & 0.9585 & OoM & 25194 \\ \hline
    With symmetry   & 0.9585 & 16534s & 12642\\
      \hline
   \end{tabular}
 \end{center}
 \label{SPIN_EFFICIENCY_2}
\end{table}



\subsubsection{Optimal trajectory reconstruction using Algorithm~$(A_2)$~\label{ETLABETE}}
Contrary to the case of the integrator~\eqref{INTEG} there is no guarantee of uniqueness for solutions of Problem~\eqref{opt:relaxed_bis} associated to Problem~\eqref{opt}, so that we propose to apply Algorithm~$(A_2)$ using symmetries via the application of successive SDPs: first~\eqref{sdp_sym_2} in order to  obtain a lower approximation $\rho_k^G$ of the optimal cost, then~\eqref{sdp_selection_sym} to recover the image of optimal trajectories by invariant polynomials. 
The efficiency is compared to the dense case which does not take into account symmetries implemented in Appendix~\ref{without_sym_rec}.
 A basis of homogeneous generating invariant polynomials is given for the $x$-component by $P_1(x_1,x_2,x_3)=x_1^2$, $P_2(x_1,x_2,x_3)=x_2^2$, $P_3(x_1,x_2,x_3)=x_1x_2$, $P_4(x_1,x_2,x_3)=x_3$, and for the $u$-component by $V(u)=u^2$.
In this setting, the polynomial system~\eqref{INVERT_LIFT} restricted to the state trajectory writes $y_j(t)=P_j(x(t))$, for every $j\in \{1,\dots,4\}$, and $t\in [0,T]$.
By application of the SDP~\eqref{sdp_selection_sym} in Algorithm~$(A_2)$, we recover approximately the moments of an extreme point $(\mu^\star,\mu_T^\star)$ in the set~$\mathcal{M}^G$ associated with Problem~\eqref{opt:relaxed}. 
We can then approximate the moments of the curves $x_1(t)^2$, $x_2(t)^2$, $x_3(t)$, for optimal trajectories $(x_1(t),x_2(t),x_3(t))$ for Problem~\eqref{opt}. 
On Figure~\ref{RECOCO}, we have plotted the reconstruction of the latter curves thanks to the Christoffel-Darboux kernels for $d=2k=10$, for which the computation time is $81.1s$. 
The curves corresponding to optimal trajectories are plotted in dashed lines.

\subparagraph{Solving $(P_1)$ as third step of Algorithm~$(A_2)$ and selection of trajectories}
Optimal trajectories of Problem~\eqref{opt} then have to be selected among the Lipschitz curves which are solutions of
$y_j(t)=P_j(x(t))$, for every $j\in \{1,\dots,4\}$, and $t\in [0,T]$, as described in Section~\ref{INV_POLY_USE} with Equation~\eqref{INVERT_LIFT}. A feasibility test can be done, which consists in solving approximately the linear programs~\eqref{opt:relaxed_u} via the SDP~$(Q_{k}^u)$ from Appendix~\ref{CONTROL_RECO}, choosing among the possible trajectories $x(t)$ allowed by the previous step.
We checked numerically in~\url{TRAJECTORY_TEST.m} the infeasibility of the curve whose $x_2$-component is equal to $\pm \sqrt{y_2(t)}$, by checking SDP infeasibility of~$(Q_{k}^u)$ with $d=6$. This fact ensures that the $x_2$-component of optimal trajectories change sign and are composed, up to a global sign change, of two branches $t\mapsto +\sqrt{y_2(t)}$ for $t\in [0,t^\star]$ then $t\mapsto -\sqrt{y_2(t)}$ for $t\geq t^\star$  where $t^\star>0$, as claimed by the structure of optimal solutions plotted on Figure~\ref{syncadm_opti}.
Note however that this step requires the knowledge of the moments of the function $t\mapsto \pm\sqrt{y_2(t)}$, which can be done thanks to computing approximate moment integrals of the square root of Christoffel-Darboux approximant of $y_2(t)$. In order to avoid this difficulty and possible accumulation of errors, the computations have been made by computing the corresponding moments thanks to the knowledge of the theoretical optimal trajectory.


\subparagraph{Control reconstruction via solving $(P_2)$ as third step of Algorithm~$(A_2)$}
Concerning the recovery of the optimal control, the same discontinuity problem occurs as for the integrator example of Section~\ref{INTEGRATOR}.
We propose to  solve the linear program and \eqref{opt:lift} via the corresponding moment relaxation~\eqref{sdp_lift}.
On Figure~\ref{syncadm1}, we illustrate the results obtained by applying the SDP~\eqref{sdp_lift}, allowing to recover approximately an optimal control $u(t)$ via Christoffel-Darboux kernels, which is plotted on Figure~\ref{split:1}. Due to moment substitutions, the computation time of~\eqref{sdp_lift} with $d=2k=10$ is equal to $118s$, with $3040$ pseudo-moment variables, which is improved w.r.t. the dense case.
The control is then plugged in the dynamics~\eqref{sistema1}, and we obtain the corresponding trajectories on Figure~\ref{adm:fig:3} and \ref{adm:fig:4}.

\begin{remark}
Note that the time needed for the second step of Algorithm~$(A_2)$ would have been approximately twice the time of its first step using Yalmip modeling, as the two steps are two symmetry-reduced problems of the same size.
However, the third step of Algorithm~$(A_2)$ may require supplementary time when solving Problem~$(P_2)$ via~\eqref{sdp_lift}, due to a reduction which only involves moment substitutions.
\end{remark}

\begin{figure}[!htbp]
    \centering
    \subfigure[Reconstruction of $(x_1(t))^2$]{\label{spin1} \includegraphics[scale=0.38]{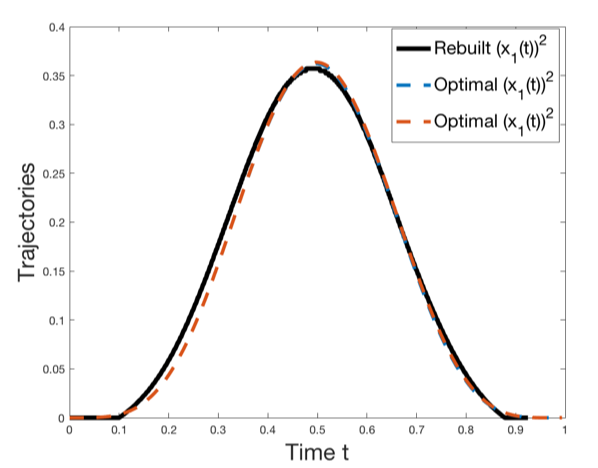}}
    \subfigure[Reconstruction of $(x_2(t))^2$]{\label{spin2} \includegraphics[scale=0.38]{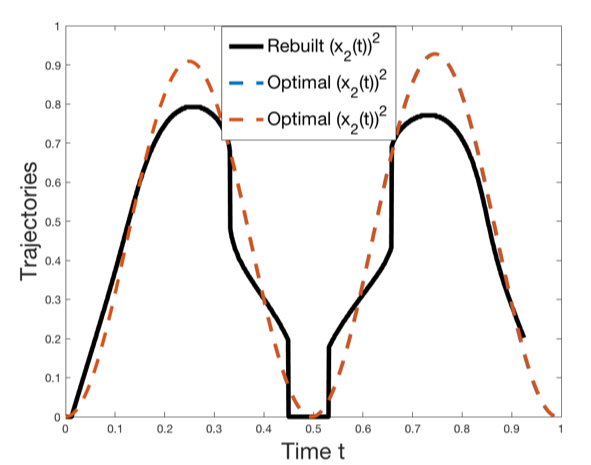}}
     \subfigure[Reconstruction of $(x_3(t))$]{\label{spin3} \includegraphics[scale=0.38]{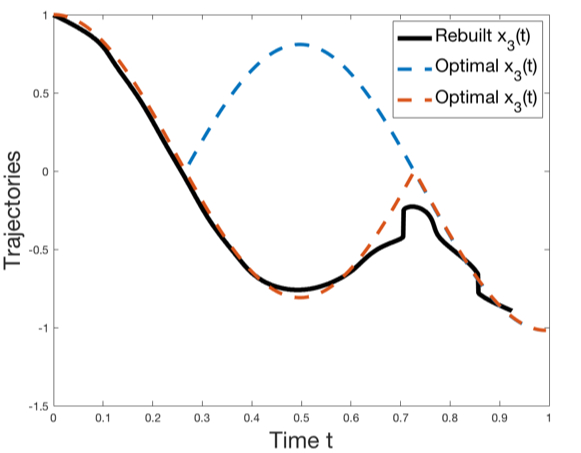}}
    \caption{Reconstruction of the image of trajectories via invariant polynomial with CD kernels, with $d=2k=10$.}
    \label{RECOCO}
\end{figure}


\begin{figure}[!htbp]
    \centering
    \subfigure[Reconstruction of the control $u(t)$ with~\eqref{sdp_lift} ]{\label{split:1} \includegraphics[scale=0.5]{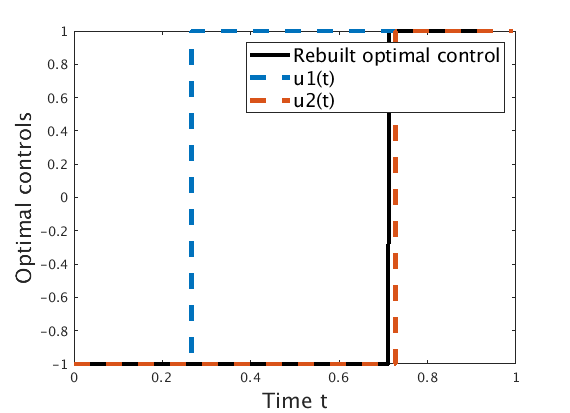}}
    \subfigure[First part of the rebuilt trajectory $x(t)$]{\label{adm:fig:3} \includegraphics[scale=0.55]{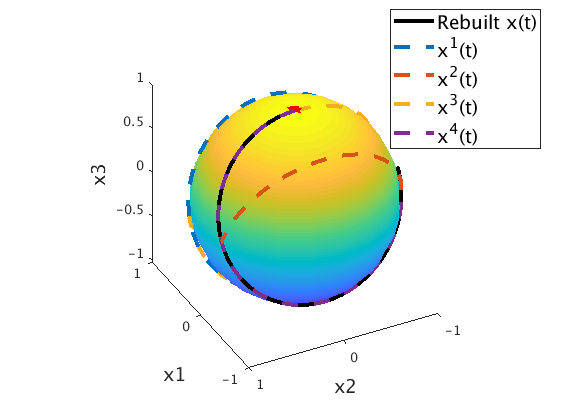}}
     \subfigure[Second part of the rebuilt trajectory $x(t)$]{\label{adm:fig:4} \includegraphics[scale=0.55]{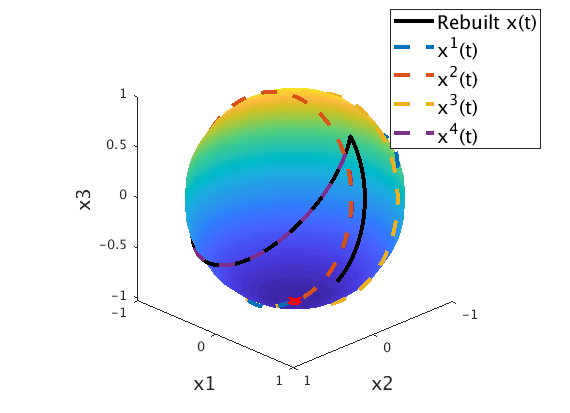}}
    \caption{Rebuilt control with~\eqref{sdp_lift} for $d=2k=10$ on Figure~\ref{split:1}. The control is plugged into the dynamics~\eqref{sistema1} on the simulations plotted on the figures~\ref{adm:fig:3} and \ref{adm:fig:4}.}
    \label{syncadm1}
\end{figure}

\section{Conclusion}
In this work, we have shown that using symmetries is an efficient way to increase the computational efficiency for solving optimal control problems using the moment-sum-of-squares hierarchy, treating both reduction of problem size as well subsequent recovery of of optimal trajectories. It turned out that symmetry reduction allows to recover the image of optimal trajectories of the state variables by invariant polynomials in a very efficient way. Then we proposed a method via moment substitution allowing to recover an optimal control.
In further works, we plan to study systems exhibiting more complex symmetries, and computer algebra methods to solve the polynomial system obtained by symmetry reduction. A longer term perspective is to study Lie group symmetries which are very common in physics.

\section*{Acknowledgments}
The authors are very grateful to Evelyne Hubert for suggesting the trajectory recovery strategy based on invariant polynomials as well as to Emmanuel Trélat for fruitful discussions. 
This work was supported by the LabEx CIMI (ANR-11-LABX-0040), the FastQI grant funded by the Institut Quantique Occitan, the PHC Proteus grant 46195TA, the European Union’s Horizon 2020 research and innovation programme under the Marie Sk{\l}odowska-Curie Actions, grant agreement 813211 (POEMA), by the AI Interdisciplinary Institute ANITI funding, through the French ``Investing for the Future PIA3'' program under the Grant agreement n${}^\circ$ ANR-19-PI3A-0004 as well as by the National Research Foundation, Prime Minister’s Office, Singapore under its Campus for Research Excellence and Technological Enterprise (CREATE) programme.

\appendix
\section{Appendix~\label{APP}}
\subsection{Proof of Proposition~\ref{INVARIANCE}~\label{PROOF_INVARIANCE}}
\proof
For the proof of the first point, let $(\mu,\mu_T)\in \mathcal{M}^G$. By $G$-invariance, we have $R(\mu)=\mu$ and $R_T(\mu_T)=\mu_T$, so that $(\mu,\mu_T)\in \{\left(R(\mu),R_T(\mu_T)\right)\mid \ (\mu,\mu_T)\in \mathcal{M}\}$.
Conversely, for $(\mu,\mu_T)\in \mathcal{M}$,  Lemma~\ref{Liouville_invariant} proves that $(R(\mu),R_T(\mu_T))\in \mathcal{M}^G$, as the group $G$ is finite.
Concerning weak-$\star$ compactness, as $R$ and $R_T$ are continuous operators for the weak-$\star$ topology, the set $\mathcal{M}^G$ is closed as the inverse image of $(0,0)$ by the continuous mapping $(\mu,\mu_T)\mapsto \left(R(\mu)-\mu,R_T(\mu_T)-\mu_T\right)$.
It follows that the set $\mathcal{M}^G$ is weak-$\star$ compact as a closed subset of the weak-$\star$ compact set $\mathcal{M}$. 

     For the proof of the second point, assume by contradiction that there exist $(\mu,\mu_T)\in \mathcal{M}$ which are not occupation measures and such that $(R(\mu),R_T(\mu_T))$ is an extreme point of $\mathcal{M}^G$. 
    By Proposition~\ref{EXTREME}, we have that $(\mu, \mu_T)$ satisfy
    \[\int_{[0,T] \times X \times U} \phi(t,x,u) d\mu(t,x,u)= \int_{S} \left(\int_{[0,T] \times X \times U} \phi(t,x,u) d\gamma(t,x,u)\right) d\nu(\gamma),\] and \[\int_{K} \varphi(x) d\mu_T(x)= \int_{S_T} \left(\int_{K} \varphi(x) d\gamma(x)\right) d\nu_T(\gamma),\] for every $(\phi,\varphi)\in C^1([0,T] \times X \times U)\times C^1(K)$, where $(\nu,\nu_T)$ are probability measures on the sets $S$ and $S_T$ of occupation measures, in the sense of Definition~\ref{OCCUP_MEAS}.
From the linearity of $R$ and $R_T$ that $R(\mu)$ and $R_T(\mu_T)$ satisfies
 \[\int_{[0,T] \times X \times U} \phi(t,x,u) d(R(\mu))(t,x,u)= \int_{S} \left(\int_{[0,T] \times X \times U} \phi(t,x,u) d(R(\gamma))(t,x,u)\right) d\nu(\gamma),\] and \[\int_{K} \varphi(x) d(R_T(\mu_T))(x)= \int_{S_T} \left(\int_{K} \varphi(x) d(R_T(\gamma))(x)\right) d\nu_T(\gamma),\] for every $(\phi,\varphi)\in C^1([0,T] \times X \times U)\times C^1(K)$.
 As $(\mu,\mu_T)\notin \mathcal{E}=S\times S_T$, we have that $(\nu,\nu_T)$ is not supported on a single point of $\mathcal{E}$. We obtain a contradiction since $(R(\mu),R_T(\mu_T))$ is assumed to be an extreme point of $\mathcal{M}^G$.
 
 The last statement follows from Krein-Milmann theorem in the infinite dimensional setting~\cite[Chapter III, Section 4]{barvinok2002course}.

\subsection{Symmetric hierarchy without block diagonalization~\label{A22}}
For two $G$-invariant pseudo-moment sequences $y$ and $z$, consider the following symmetry-adapted relaxation hierarchy, indexed by $k\geq k_0$, where $k_0$ is chosen as in Section~\ref{LET_DENSE}, denoted by~\eqref{sdp_sym_1}:
\begin{equation} \tag{$Q_k^I$}\label{sdp_sym_1}
\begin{aligned}
r_k := \inf\limits_{y,z} \quad & L_z(h)+L_y(H)\\
\textrm{s.t.} \quad &   M_k(y),M_k(z)\succeq 0\\
   & M_{k-\lceil\text{deg}(v_j)/2\rceil}(v_j z(x))\succeq 0  \\
   &M_{k-\lceil\text{deg}(w_j)/2\rceil}(w_j z(u))\succeq 0 \\
  &M_{k-\lceil\text{deg}(\theta_j)/2\rceil}(\theta_j y)\succeq 0 \\
  &M_{k-1}(t(1-t) z(t))\succeq 0 \\
  & L_y(\phi)-L_z\left(\partial \phi/ \partial t+\langle \nabla_x \phi, f \rangle \right)=\phi(0,x_0), \ \forall \phi=(t^s x^\alpha) \in \R[t,x] \ \text{s.t.} \ s+|\alpha|\leq 2k+1-\text{deg}(f)\\
\end{aligned}
\end{equation}
The latter semi-definite conditions being equivalent to those of~\eqref{sdp_sym_2}, we have the equality $r_k=\rho_k^G$, for every $k\geq k_0$.

\subsection{Dynamical test of the solutions of Equation~\eqref{INVERT_LIFT} and control recovery~\label{CONTROL_RECO}}
 In general, not every Lipschitz curve solving Equation~\eqref{INVERT_LIFT} in its $x$-component is feasible nor optimal for Problem~\eqref{optimalcontrol}.
It is then a crucial task to propose a numerical method testing the feasibility and optimality of $x(\cdot)$ for Problem~\eqref{optimalcontrol} and in the positive case recovering the associated optimal control $u(\cdot)$.

 \begin{assumption}\label{KNOW_TRAJ}
 We know a Lipschitz curve $x(\cdot)$ solving Equation~\eqref{INVERT_LIFT} in its $x$-component and/or its associated moments w.r.t. the Lebesgue measure on $[0,T]$. 
 \end{assumption}

 For $\mu\in \mathscr{M}_+([0,T]\times U)$, consider the following reduced Liouville equation
 \begin{equation}\label{Liouville_u}
\phi(T,x(T)) -\phi(0,x_0) = \int_{[0,T]\times \mathbb{R}^n\times U} \frac{\partial \phi}{\partial t} + \left\langle \nabla_x\phi(t,x(t)), f(x(t),u)\right\rangle\,d\mu(t,u),
\end{equation}
for every $\phi\in C^1([0,T]\times X)$.
 
 Introduce the following linear program:
 \begin{equation} \label{opt:relaxed_u}
\begin{aligned}
\inf\limits_{\mu} \quad & \int_{[0,T]\times U} h(t,x(t),u)d\mu(t,u)\\
\textrm{s.t.} \quad &   \mu \text{\;\,satisfy (\ref{Liouville_u})}\\
  &\mu\in \mathscr{M}_+([0,T]\times U).  \\
\end{aligned}
\end{equation}
Note that, as the curve $x(\cdot)$ is fixed, the term $H(x(T))$ involved in the cost of Problem~\eqref{optimalcontrol} can be removed, and if there exists $u(\cdot)$ such that $(x(\cdot),u(\cdot))$ is feasible and optimal for Problem~\eqref{optimalcontrol}, then the set of solutions of Problem~\eqref{opt:relaxed_u} is non-empty and contains the occupation measure $d\nu(t,u)=\delta_{u(t)}(du)dt$.

In order to solve Problem~\eqref{opt:relaxed_u}, we propose the following hierarchy, 
indexed by $k\geq k_0$, for pseudo moment sequences $z=(z_\gamma)_{\gamma\in \mathcal{B}}$ and $y=(y_\alpha)_{\alpha\in \tilde{\mathcal{B}}}$, where $\mathcal{B}$ is a monomial basis of $\R[t,x,u]$, and $\tilde{\mathcal{B}}$ is a monomial basis of $\R[x]$.
It is denoted by $Q_k^u$ and defined as a slight modification of~\eqref{sdp_dense}, by adding the conditions $L_z(t^s x^\alpha)=L_{z^\star}(t^s x^\alpha), \forall s,\alpha \ \text{s.t.} \ s+|\alpha| \leq 2k$ 
and $L_y(x^\alpha)=L_{y^\star}(x^\alpha), \forall \alpha\in \N^n \ \text{s.t.} \ |\alpha| \leq 2k$, where $z^\star, y^\star$ are the moment sequences associated with the measures $d\nu(t,x)=\delta_{x(t)}(dx) dt$ and $d\nu_T(x)= \delta_{x(T)}(dx)$, that are known a priori under Assumption~\ref{KNOW_TRAJ}.

We have the following properties:
\begin{itemize}
 \item If there exists $k\geq k_0$ such that $(Q_{k}^u)$ is unfeasible, then the curve $x(\cdot)$ is unfeasible for Problem~\eqref{optimalcontrol}, and the solution has to be rejected. 
 \item In the case where $(Q_{k}^u)$ is feasible, check optimality, by checking that $\inf Q_k^u + H(x(T))-\rho_k^G$ is close to 0 when $k\to +\infty$. 
 \end{itemize}

\begin{remark}
In order to avoid approximation issues, in our numerical implementations we replace the constraints $L_z(t^s x^\alpha)=L_{z^\star}(t^s x^\alpha), \forall s,\alpha \ \text{s.t.} \ s+|\alpha| \leq 2k$ 
and $L_y(x^\alpha)=L_{y^\star}(x^\alpha), \forall \alpha\in \N^n \ \text{s.t.} \ |\alpha| \leq 2k$
by
$|L_z(t^s x^\alpha)-L_{z^\star}(t^s x^\alpha)|\leq \epsilon, \forall s,\alpha \ \text{s.t.} \ s+|\alpha| \leq 2k$,
 $|L_y(x^\alpha)-L_{y^\star}(x^\alpha)|\leq \epsilon, \forall \alpha\in \N^n \ \text{s.t.} \ |\alpha| \leq 2k$,  with $\epsilon>0$ small enough.
\end{remark}

\subsection{Optimal trajectory reconstruction without using symmetries~\label{without_sym_rec}}
We propose to achieve the reconstruction of trajectories without using symmetries described in Section~\eqref{LINEAR_FUNK}. 
The SDP~\eqref{sdp_selection1} allows finding the approximate values of the moments of an optimal trajectory for Problem~\eqref{opt} in the variables $(t,x,u)$.
We have plotted the rebuilt optimal control $u(t)$ and trajectory components $(x_1(t),x_2(t),x_3(t))$ on Figure~\ref{syncadm_opti_rec}, to be compared with the optimal solutions obtained with the four optimal controls $(u_1(t),-u_1(t),u_2(t),-u_2(t))$, which are plotted in dashed lines. We notice that the control curve is close to the optimal one, contrary to the recovered trajectories, especially the $x_2$-component.
\begin{figure}[!ht]
    \centering
    \subfigure[Reconstruction of the control $u(t)$]{\label{ROUGEb} \includegraphics[scale=0.55]{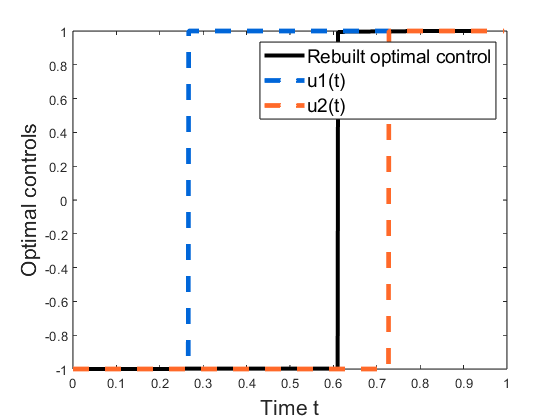}}
    \subfigure[Reconstruction of $x_1(t)$]{\label{COULEb} \includegraphics[scale=0.55]{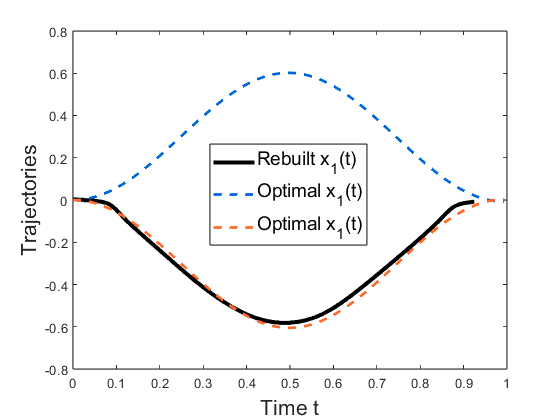}}
     \subfigure[Reconstruction of $x_2(t)$]{\label{SOEURb} \includegraphics[scale=0.55]{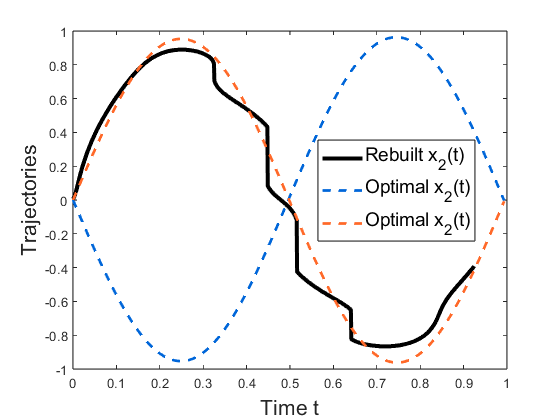}}
       \subfigure[Reconstruction of $x_3(t)$]{\label{MEREb} \includegraphics[scale=0.55]{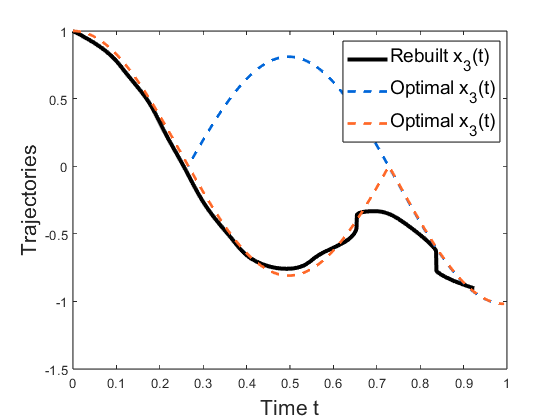}}
    \caption{Reconstruction with (CD) kernel of optimal control and trajectories, without symmetry for Problem~\eqref{opt}, with $d=2k=10$, compared to theoretical optimal solutions in dashed lines.}
    \label{syncadm_opti_rec}
\end{figure}
For $d=2k=10$, we obtain the reconstruction of Figure~\ref{apost} by plugging the control obtained on Figure~\ref{ROUGEb} into the dynamics of Equation~\eqref{sistema1}. On the figure, the simulated trajectory is plotted in black, and the optimal trajectories corresponding to the controls $(u_1(t),-u_1(t),u_2(t),-u_2(t))$ are plotted in dashed lines. 
\begin{figure}[!ht]
    \centering
    \subfigure[First part of the rebuilt trajectory $x(t)$]{\label{reconstr} \includegraphics[scale=0.55]{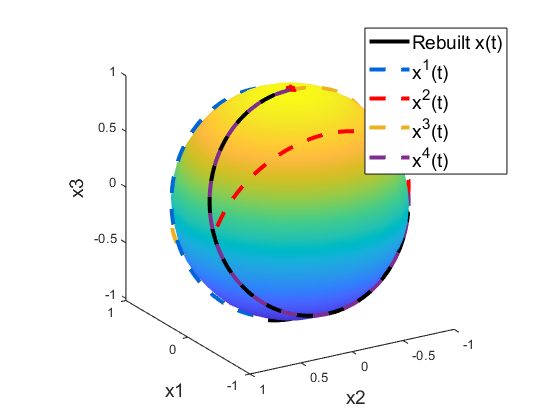}}
    \subfigure[Second part of the rebuilt trajectory $x(t)$]{\label{recon} \includegraphics[scale=0.55]{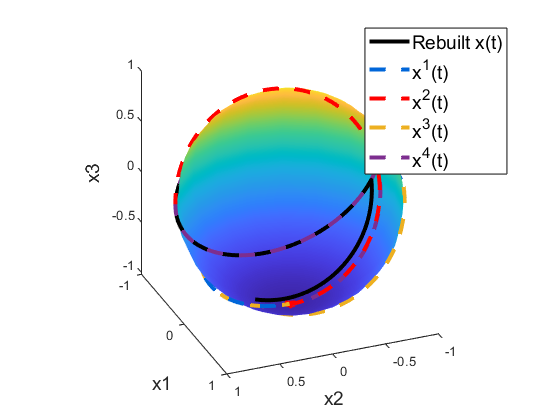}}
    \caption{Simulation of the trajectory of~\eqref{opt} in black by plugging the rebuilt control by SDP~\eqref{sdp_selection1} into the dynamics~\eqref{sistema1}, without symmetry, with $d=2k=10$, compared to theoretical optimal trajectories in dashed lines.}
    \label{apost}
\end{figure}

\bibliographystyle{abbrv}
\bibliography{OCP_main}

\end{document}